\documentclass[11pt,a4paper]{article}

\usepackage[T1]{fontenc}
\usepackage[margin=1in]{geometry}
\usepackage{amscd}
\usepackage{amssymb}
\usepackage{amsmath}
\usepackage{amsthm}
\usepackage{enumitem}

\usepackage[backend=biber,
style=trad-abbrv,
            backref=true,
            backrefstyle=three,
            hyperref=true,
	    url=false,
	    isbn=false,
	    doi=false,
	    sortcites]{biblatex}
\usepackage{pgfplots}
\pgfplotsset{compat=1.16}
\usepackage{wrapfig, caption}
\usepackage{float}
\usetikzlibrary{arrows}
\usepackage[font=small,labelfont=bf]{caption}
\usepackage{xcolor}
\usepackage{mathtools}
\usepackage[colorlinks=true, linkcolor=blue, citecolor=blue]{hyperref}
\usepackage{comment}
\usepackage{etoolbox}

\usepackage[capitalise]{cleveref}
\crefname{enumi}{}{}
\crefname{theorem}{Theorem}{Theorems}
\crefname{corollary}{Corollary}{Corollaries}
\crefname{proposition}{Proposition}{Propositions}
\crefname{lemma}{Lemma}{Lemmas}
\crefname{claim}{Claim}{Claims}
\crefname{conjecture}{Conjecture}{Conjectures}
\crefname{problem}{Problem}{Problems}
\crefname{question}{Question}{Questions}
\crefname{definition}{Definition}{Definitions}
\crefname{remark}{Remark}{Remarks}
\crefname{example}{Example}{Examples}
\makeatletter
\newcommand{\crefpatch}[1]{%
  \AtBeginEnvironment{#1}{%
    \let\orig@label\label
    \renewcommand{\label}[1]{%
      \ifdefstring{\@currentcounter}{theorem}
        {\orig@label[#1]{##1}}%
        {\orig@label{##1}}%
    }%
  }%
}
\crefpatch{corollary}
\crefpatch{proposition}
\crefpatch{lemma}
\crefpatch{claim}
\crefpatch{conjecture}
\crefpatch{problem}
\crefpatch{question}
\crefpatch{definition}
\crefpatch{remark}
\crefpatch{example}
\makeatother

\newtheorem{theorem}{Theorem}[section]
\newtheorem*{theorem*}{Theorem}

\newtheorem*{theoremY*}{Theorem Y}

\newtheorem*{theoremAB*}{Theorem AB}

\newtheorem{corollary}[theorem]{Corollary}
\newtheorem*{corollary*}{Corollary}
\newtheorem{proposition}[theorem]{Proposition}
\newtheorem{lemma}[theorem]{Lemma}

\newtheorem*{claim*}{Claim}

\theoremstyle{definition}
\newtheorem{definition}[theorem]{Definition}
\theoremstyle{remark}
\newtheorem{remark}[theorem]{Remark}
\newtheorem*{remark*}{Remark}

\renewcommand{\Bbb}[1]{\mathbb{#1}}

\newcommand{\bbN}{{\Bbb N}}         

\newcommand{\bbR}{{\Bbb R}}        

\newcommand{\0}{\mathbf{0}}

\newcommand{\ie}{{\it i.e.}\/ }

\newcommand{\eps}{\varepsilon}

\newcommand{\bi}{\mathbf{i}}
\newcommand{\bj}{\mathbf{j}}

\newcommand{\bk}{{\mathbf{k}}}
\newcommand{\bh}{{\mathbf{h}}}
\newcommand{\bu}{\mathbf{u}}

\allowdisplaybreaks

\numberwithin{equation}{section}

\title{On exponential separation of analytic self-conformal sets on the real line}

\author{Bal\'azs B\'ar\'any\\(BME) \and Istv\'an
  Kolossv\'ary\\ (R\'enyi Institute) \and Sascha
Troscheit \\(Uppsala)}
\date{\today}

\newcommand{\Addresses}{{
  \bigskip
  \footnotesize

  B.~B\'ar\'any, \textsc{Department of Stochastics, HUN-REN-BME Stochastics Research Group,
  Institute of Mathematics, Budapest University of Technology and Economics, M\H{u}egyetem rkp.~3.,
H-1111 Budapest, Hungary.}\par\nopagebreak
  \textit{E-mail address}, B.~B\'ar\'any: \texttt{barany.balazs@ttk.bme.hu}

  \medskip

  I.~Kolossv\'ary, \textsc{HUN-REN Alfr\'ed R\'enyi Institute of Mathematics, 1053 Budapest, Re\'altanoda u.
13–15, Hungary}\par\nopagebreak
  \textit{E-mail address}, I.~Kolossv\'ary: \texttt{istvanko@renyi.hu}
   \medskip

  S.~Troscheit, \textsc{Department of Mathematics, Box 524, 751 20 Uppsala University, Sweden.}\par\nopagebreak
  \textit{E-mail address}, S.~Troscheit: \texttt{sascha.troscheit@math.uu.se}
}}

\begin{document}

\frenchspacing
\maketitle


\begin{abstract}
In a recent article, Rapaport showed that there is no dimension drop for
exponentially separated analytic IFSs on the real line. We show that the set of such exponentially
separated IFSs in the space of analytic IFSs contains an open and dense set in the $\mathcal{C}^2$
topology. Moreover, we give a sufficient condition for the IFS to be exponentially separated which
allows us to construct explicit examples which are exponentially separated. The key
technical tool is the introduction of the \emph{dual IFS} which we believe has significant interest
in its own right. As an application we also characterise when an analytic IFS can be conjugated to a
self-similar IFS. 
\end{abstract}

%
%

\section{Introduction and main results}
The geometric properties of attractors of iterated function systems have been extensively studied in recent
decades.
An iterated function system (IFS) $\Phi$ is a finite collection of strictly contracting self-maps
$(f_i)_{i\in\mathcal{I}}$ on a complete separable metric space $X$. Throughout, we assume
$N\coloneqq\#\mathcal{I}\geq2$. By a result of
Hutchinson~\cite{Hutchinson_Attractor_81}, there exists a unique non-empty compact set $\Lambda$
satisfying the
invariance
\begin{equation}\label{eq:Hutchinson}
  \Lambda = \bigcup_{i\in\mathcal{I}} f_i(\Lambda).
\end{equation}
Similarly, one can consider measures invariant under this relation in the following sense.
Given an IFS and a non-degenerate probability vector $\mathbf{p}=(p_i)_{i\in\mathcal{I}}$, \ie
$\sum_{i\in\mathcal{I}}p_i=1$ and all $p_i>0$, there exists a unique Borel probability measure
$\mu_{\mathbf{p}}$ supported on $\Lambda$ such that
\begin{equation}\label{eq:conformaMeasure}
\mu_{\mathbf{p}}=\sum_{i \in \mathcal{I}} p_i \cdot \mu_{\mathbf{p}} \circ f_i^{-1}.
\end{equation}
The size of such sets and measures, as measured through dimension, is one of the main focus points of fractal
geometry. Answering questions in such generality is generally unfeasible and one often restricts to
the simpler setting of $X=\mathbb{R}^d$ and where the $f_i$ are simpler mappings such as
similarities, affinities, or conformal maps.

Hutchinson \cite{Hutchinson_Attractor_81} considered the Hausdorff dimension of the attractor of
IFSs consisting of similarities on $\mathbb{R}^d$ under a separation condition, the open set
condition (OSC), and provided a formula for the dimension of the sets depending solely on the
contraction ratios of the similarities. In particular, the OSC holds under the stronger assumption
of the strong separation condition (SSC). We say that the IFS $\Phi=(f_i)_{i\in \mathcal{I}}$
satisfies the SSC if
\begin{equation*}
  f_i(\Lambda)\cap f_j(\Lambda) = \varnothing\quad\text{whenever}\quad i\neq j.
\end{equation*}
Hutchinson also showed that this value provides an upper bound for
the dimension regardless of the overlaps of images of $\Lambda$ under the $f_i$.  Bowen
\cite{Bowen75} and Ruelle \cite{Ruelle2004} extended this result to attractors of
$\mathcal{C}^{1+\alpha}$ conformal mappings, with the natural upper bound known as the conformality
dimension. Similar questions can be asked for the measure, see Cawley and Mauldin
\cite{CawleyMauldin92}, and Patschke \cite{Patzschke97}, for the self-similar and self-conformal
setting. 

The actual value of the Hausdorff dimension of the attractor may drop below this natural upper
bound. This occurs, for instance, when the maps have exact overlaps and it is an open problem (the
dimension drop conjecture) whether this is the only mechanism for a dimension drop to occur, see
Simon \cite{Simon1996}.  It was since verified that the conformality dimension coincides with the
Hausdorff dimension, at least typically, for several notions of typicality.

Simon, Solomyak, and Urba\'nski \cite{SimonSolomyakUrbanski2,SimonSolomyakUrbanski1} considered
parametrised families of $C^{1+\alpha}$ IFSs and showed that under some technical assumptions
(transversality condition) there is no dimension drop for the set and measure for almost every
parameter with respect to the Lebesgue measure. Relying on transversality methods, Simon and
Solomyak \cite{SimonSolomyak02} showed that for self-similar sets the set of exceptions where a
dimension drop occurs is a meagre set in the Baire category sense. 

Other approaches, involving different notions of dimension or separation conditions were also
considered, see
Zerner \cite{Zerner1996}; Lau and Ngai \cite{LauNgai1999}; Ngai and Wang \cite{NgaiWang2001}; Fraser,
Henderson, Olson, and Robinson \cite{FHOR2015}; and Angelevska, K\"aenm\"aki, and Troscheit
\cite{AngelevskaKaenmakiTroscheit2020}; and references therein.

A major breakthrough was made by Hochman \cite{Hochman_SelfSimESC_Annals} who showed that no
dimension drop occurs for self-similar sets and measures on the line under the
assumption of the exponential separation condition (ESC).
The ESC is a condition that is satisfied under many natural assumptions and also holds for many typical
systems including those described above. 
For instance, it was shown in~\cite{Hochman_SelfSimESC_Annals}
that if the parameters defining the self-similar IFS $\Phi$ are algebraic and $\Phi$ has no exact
overlaps then the ESC holds. The algebraic condition has since been relaxed in certain settings,
see~\cite{FengFeng_DimHomoIFSAlgebraicTrans, Rapaport_ExactOverlapsAlgebraicContr,
RapaportVarju_Duke24, Varju_BernConv_Annals19}.
Generalisations have also been made to higher dimensions by Hochman \cite{Hochman2017} and for
special non-linear maps, M\"obius transformations, by Hochman and Solomyak
\cite{HochmanSolomyak_Invent17}.

Recently, Rapaport \cite{Rapaport_SelfConfESC25arXiv} extended the concept of the exponential
separation condition to general
analytic self-conformal iterated function systems. Furthermore, under this condition, Rapaport
showed that the dimension of the set and the measure does not drop.
Rapaport pointed out that the analyticity assumption is crucial for the proof of his main result.
In fact, his methods fail even for $C^\infty$-maps.
For analytically parametrised systems of analytic self-conformal IFSs, Rapaport verified that the
ESC holds for almost every choice of parameter in the sense of Hausdorff dimension.
However, he could not provide any concrete examples other than those already known.

The main purpose of this article is to provide verifiable sufficient conditions that guarantee that the
ESC holds. Further, we show that this property holds for an open and dense set of IFSs with respect to the
$\mathcal{C}^2$ topology. We will also explore when analytic self-conformal IFS can be conjugated to
self-similar systems.
The methods involve constructing a \emph{dual} IFS, which we believe is of independent interest. 
Just like in the case of Rapaport, the assumption of analyticity is crucial for our methods as well,
since at various places in the proof we need to transition from local information to global bounds.
To ensure sufficient regularity, the analyticity is a natural condition, see for example
\cite{AlgomEtal_NonLinHyperbolicIFS} and \cite[Example~7.7]{AngelevskaKaenmakiTroscheit2020}.

\paragraph{Lean development.} The main results in this paper have been formalised in Lean~4 using
Mathlib, and the development is registered on Palomar~\cite{BaranyKolossvaryTroscheit_Palomar26}.
In particular, we formalised the genericity and
separation results in \cref{thm:ESCOpenDense,thm:main,prop:example,thm:DualSSC} and the conjugation
results in \cref{thm:SubConjugation,thm:DualConj}, together with the auxiliary lemmas and propositions.
We also verify the SESC for the polynomial three-map example in \cref{sec:examples}. The dimension
statements \cref{thm:RapaportMain,cor:dim} and the dimension calculations in
\cref{sec:examples}, including the local and $L^q$ dimensions of the natural measure, are proved
conditional on explicitly stated results from the literature. That is, we did not formalise 
Bowen's existence theorem for Gibbs measures \cite[Theorem~1.4]{Bowen75} and Rapaport's dimension formulae
\cite[Theorem~1.2 and Corollary~1.3]{Rapaport_SelfConfESC25arXiv}. Further, Feng and Hu's exact-dimensionality result
\cite[Theorem~2.8]{FengHu_DimensionIFS} is stated as a hypothesis for the measure conclusion.
Some formal proofs use slightly different arguments from those presented here.

\paragraph{Acknowledgements.} The authors would like to thank Amir Algom for pointing out various
useful references on conjugation of analytic IFSs to self-similar IFSs. We further thank Ga\'etan
Leclerc and Tuomas Sahlsten for further discussions and Antti K\"aenm\"aki for pointing out
corrections to an earlier version of this article. Finally, we thank the anonymous
referees for their careful reading and helpful comments.

BB acknowledges support from grant NKFI~K142169, and grant NKFI
KKP144059 \textit{Fractal geometry and applications} Research Group. IK is supported by the European
Research Council Marie Sk\l odowska-Curie Actions Postdoctoral Fellowship $\#101109013$ and the
Hungarian NRDI Office grants K142169 and STARTING152587. ST acknowledges funding from the Lisa
\& Carl-Gustav Esseen's Mathematics Fund.

\paragraph{AI statement.}
We declare the use of AI tools in accordance with the Leiden Declaration on Artificial Intelligence
and Mathematics~\cite{LeidenDeclaration2026}.
LLM tools were used for some of the copy-editing of the material in this article as well as the
figures, but not in the intellectual
creation of the underlying ideas. The Lean
formalisation was mostly done by LLM agents.
The authors reviewed the generated
material, and the formal Lean proofs were additionally machine-checked using Comparator and registered
on Palomar~\cite{BaranyKolossvaryTroscheit_Palomar26}.
The authors take full responsibility for the article and the accompanying formalisation.

\subsection{Notations and main results}
For convenience we write $I=[0,1]$.
For $\eps>0$ we define the class $\mathcal{S}^\omega_\eps(I)$
consisting of maps $f\colon I\to\mathbb{R}$ with the following properties:
\begin{enumerate}[label=(\Alph*)]
  \item\label{it:a} $f$ extends analytically to the open complex $2\eps$-neighbourhood
    $\mathcal{B}_{2\eps}$ of  $I$, where $\mathcal{B}_{\delta}\coloneqq\{z\in\mathbb{C}:\exists
    x\in I,\, |z-x|<\delta\}$,
  \item\label{it:b} $f(I) \subseteq I$ and $f(\overline{\mathcal{B}_{\eps}}) \subseteq
    \mathcal{B}_{\eps}$,
  \item\label{it:c} $0<|f'(x)|<1$ for all $x\in\overline{\mathcal{B}_{\eps}}$.
\end{enumerate}
We write $\mathcal{S}^\omega(I)\coloneqq\bigcup_{\eps>0}\mathcal{S}^\omega_\eps(I)$.
Clearly, $\mathcal{S}^\omega_\eps(I)\subseteq\mathcal{S}^\omega_{\eps'}(I)$
whenever $0<\eps'\leq\eps$.
Let
\[
  d_2(f,g) = \sup_{x\in I} |f(x)-g(x)|+\sup_{x\in I} |f'(x)-g'(x)| +\sup_{x\in I} |f''(x)-g''(x)|
\]
for the $\mathcal{C}^2$ metric and equip the space $\mathcal{S}^\omega(I)$ with the
$\mathcal{C}^2$ topology induced by $d_2$.
For each $N\geq2$, we consider the space $\mathfrak{S}_N$ of IFSs $\Phi=(f_i)_{i=1}^N$ with
$f_i\in\mathcal{S}^\omega([0,1])$. Whenever we work with a fixed IFS $\Phi\in\mathfrak{S}_N$, we
choose $\eps>0$ such that $f_i\in\mathcal{S}^\omega_\eps(I)$ for every $i$. We call such a
radius \emph{admissible} for $\Phi$. This choice may depend on $\Phi$.
With slight abuse of notation, let
$$
d_2((f_i)_{i=1}^N,(g_i)_{i=1}^N)\coloneqq\max_{i\in\mathcal{I}}d_2(f_i,g_i)
$$
be the $\mathcal{C}^2$ metric on the space of IFSs $\mathfrak{S}_N$. Note that we consider an IFS to
be an ordered tuple of maps, hence the single maximising index.

From the finite alphabet $\mathcal{I}\coloneqq\{1,\ldots,N\}$ we construct infinite words
$\bi=(i_1,i_2,i_3,\dots) \in\Sigma \coloneqq \mathcal{I}^{\bbN}$ and finite words
$\bi=(i_1,\ldots,i_k)\in \Sigma_k\coloneqq \mathcal{I}^k$ of length $k\in\mathbb{N}$, where
$\Sigma_0=\{\emptyset\}$ is just the empty word. The length of $\bi\in\Sigma_k$ is $|\bi|=k$ and for
$\bi\in\Sigma$ it is $|\bi|=\infty$. We denote the set of all finite words by $\Sigma_* =
\bigcup_{k=0}^\infty \Sigma_k$. For $\bi\neq\bj\in\Sigma\cup\Sigma_*$ we write $\bi\wedge\bj$
to denote the longest $\mathbf{k}\in\Sigma_*$ such that $\mathbf{k} = (i_1,\dots,
i_{|\mathbf{k}|})=(j_1,\dots,j_{|\mathbf{k}|})$.
For any finite word $(i_1,\dots,i_n)\in\Sigma_*$, we write
\[
f_{i_1,\dots,i_n}\coloneqq f_{i_1}\circ \dots \circ f_{i_n},
\]
where by convention $f_{\emptyset}=\mathrm{Id}$. The natural projection $\pi\colon\Sigma\to\bbR$ defined by
\begin{equation}\label{eq:natProj}
\pi(\bi)\coloneqq\lim_{n\to\infty}f_{i_1,\dots,i_n}(0)
\end{equation}
satisfies $\pi(\Sigma)=\Lambda$, where $\Lambda$ is the attractor satisfying the invariance in
\cref{eq:Hutchinson}. 
Following Rapaport \cite{Rapaport_SelfConfESC25arXiv}, we formally define the exponential separation
condition for analytic self-conformal IFSs.
\begin{definition}
  We say that the IFS $\Phi=(f_i)_{i\in\mathcal{I}}$ 
  satisfies the \emph{exponential separation condition (ESC)} if there exists $c>0$ such that for
  infinitely many $n\in\mathbb{N}$,
  \[
    \sup_{x\in[0,1]} |f_{\bi}(x)-f_{\bj}(x)| \geq c^n
  \]
  for all distinct $\bi,\bj\in\Sigma_n$. If the ESC does not hold,
  then we say that $\Phi$ has \emph{super-exponential condensation}.	
\end{definition}
\begin{definition}\label{def:SESC}
  We say that the IFS $\Phi=(f_i)_{i\in\mathcal{I}}$ 
  satisfies the \emph{strong exponential separation condition (SESC)} if there exists $c>0$ such that
  \[
    \sup_{x\in[0,1]} |f_{\bi}(x)-f_{\bj}(x)| \geq c^n
  \]
  for all $\bi,\bj\in\Sigma_n$ and $n\in\bbN$, where $\bi\neq\bj$. 
  If the SESC does not hold, \ie there exist a sequence $(\eta_n)_n$ and a subsequence
  $n_{\ell}\in\bbN$ and distinct words $\bi,\bj\in\Sigma_{n_\ell}$ such that
  $\log(\eta_n)/n\to-\infty$ and
  \[
    \sup_{x\in[0,1]}|f_{\bi}(x)-f_{\bj}(x)| \leq \eta_{n_\ell},
  \]
  then we say that $\Phi$ has \emph{weak super-exponential condensation}
\end{definition}
\begin{definition}
  We say that the IFS $\Phi=(f_i)_{i\in\mathcal{I}}$ 
 has \emph{exact overlaps} if there exist distinct $\bi,\bj\in\Sigma_*$ such that
 $f_{\bi}(x)=f_{\bj}(x)$ for every $x\in\Lambda$.
\end{definition}

We remark that, for IFSs with non-singleton attractors, the separation conditions satisfy the
implications
\[
  \text{SSC $\Rightarrow$ SESC $\Rightarrow$ ESC
  $\Rightarrow$ no exact overlaps.} 
\]

The main objective of this article is to give generic and explicit conditions under which the
ESC holds for analytic IFSs. Roughly speaking, our first
main result says that the property that an analytic IFS satisfies the SESC is a generic property in
a topological sense.

\begin{theorem}\label{thm:ESCOpenDense}
  The set of IFSs $\{\Phi \colon \Phi\text{ satisfies SESC}\}\subseteq \mathfrak{S}_N$ contains an open and
  dense subset in the $\mathcal{C}^2$ topology.
\end{theorem}
Our second main result provides a sufficient condition under which an analytic IFS satisfies the
SESC. We demonstrate in~\cref{sec:examples} that it can be used to construct completely explicit
examples of analytic IFSs that satisfy the SESC. In order to state the result, we introduce more
notation. Since we will often write finite words ``backwards'', we adopt the convention that for
$n\neq m\in\mathbb{N}$,
\begin{equation*}
	\bi_m^n\coloneqq
	\begin{cases}
		(i_m,i_{m+1},\ldots,i_{n-1},i_n), &\text{if } m<n; \\
		(i_m,i_{m-1},\ldots,i_{n+1},i_n) &\text{if } m>n,
	\end{cases}
\end{equation*}
where $|\bi|\geq\max\{m,n\}$. This will most often be used in the form $\bi_1^n=(i_1,\ldots,i_n)$ or
$\bi_n^1=(i_n,\ldots,i_1)$. Thus for compositions of maps $f_{\bi_n^1}=f_{i_n}\circ \dots \circ
f_{i_1}$. Sometimes either $m$ or $n$ is $0$. In these cases, the convention is that both $\bi_0^n$
and $\bi_m^0$ are the empty word $\emptyset$. By default, if we simply write
$\bi\in\Sigma_*\cup\Sigma$ then the subscripts are understood to be in increasing order starting
from 1 until $|\bi|$. The concatenation of two finite words is $\bi\bj$, while slightly abusing
notation $\bi^{\infty}\in \Sigma$ is the infinite word obtained by concatenating $\bi\in\Sigma_*$
infinitely many times. 
For any $\bi\in \Sigma\cup\Sigma_*$, we introduce the function 
\begin{equation}\label{eq:H_i(x)}
H_{\bi}(x)=H_{\bi_{1}^{|\bi|}}(x) \coloneqq \sum_{n=1}^{|\bi|}
\frac{f''_{i_n}}{f'_{i_n}}(f_{\bi_{n-1}^1}(x))\cdot f'_{\bi_{n-1}^1}(x),
\end{equation}
where the order of the indices is important. The quotient $f_i''/f_i'$ is the
\emph{pre-Schwarzian derivative} of $f_i$.
At this point the motivation for $H_{\bi}(x)$ may be unclear,
but will be made apparent in~\cref{sec:DualIFSFull}. 
We can now state our second main result.
\begin{theorem}
  \label{thm:main}
  Let $\Phi\in\mathfrak{S}_N$. If for all distinct $\bi,\bj \in\Sigma\cup\Sigma_*$ with $|\bi|=|\bj|$ we have
\begin{equation}\label{eq:H_iSSC}
    \sup_{x\in[0,1]} |H_{\bi}(x) - H_{\bj}(x)| > 0,
\end{equation}
then $\Phi$ satisfies the SESC.
\end{theorem}
Apart from its practical use, \cref{thm:main} is also a crucial step in
proving~\cref{thm:ESCOpenDense} and shows that the $\mathcal{C}^2$ topology
is the natural choice for typicality.
This is because \cref{eq:H_iSSC} cannot be satisfied for linear maps, having $f''_i =0$, and the
$\mathcal{C}^2$ is the finest topology for which we can distinguish strictly conformal maps
sufficiently, see also \cref{sec:ConjLinSys}.

We shall see in~\cref{thm:DualSSC} that~\cref{eq:H_iSSC} has an elegant interpretation as an analog
of the SSC on a larger space of IFSs which we further elaborate on in~\cref{sec:DualIFSFull}. 
Theorem~\ref{thm:main} is proved in~\cref{sec:ProofSufficientCond}, while the proof of
\cref{thm:ESCOpenDense} is postponed until \cref{sec:ProofESCOpenDense}.

\subsection{Discussion}
We give further context to our main results. We first discuss the dimension theoretic implications
of our main results and  give explicit examples of IFSs that satisfy the SESC. We end the section by
discussing a link to conjugation with self-similar IFSs.

\subsubsection{Dimension theoretic consequences}
Given a non-degenerate probability vector $\mathbf{p}$ and a self-conformal IFS $\Phi\in
\mathfrak{S}_N$, we define the entropy of $\mathbf{p}$ by
\begin{equation*}
H(\mathbf{p})\coloneqq -\sum_{i\in\mathcal{I}} p_i\log p_i;
\end{equation*}
the Lyapunov exponent associated to $\mathbf{p}$ and $\Phi$ by
\begin{equation*}
\chi=\chi(\Phi,\mathbf{p}):=-\sum_{i \in \mathcal{I}} p_i \int \log \big|f_{i}^{'}(x)\big|
\mathrm{d} \mu_{\mathbf{p}}(x),
\end{equation*}
where $\mu_{\mathbf{p}}$ is the self-conformal measure defined in \cref{eq:conformaMeasure}.
For $t\geq 0$, we define the pressure function
\begin{equation*}
P(t)=P_{\Phi}(t):=\lim _{n \rightarrow \infty} \frac{1}{n} \log \sum_{\bi \in
\mathcal{I}^n} \Big(\sup_{x\in[0,1]} \big|(f_{\bi_1^n})^{\prime}(x)\big|\Big)^t,
\end{equation*}
which is well defined by sub-additivity. It is convex, strictly decreasing and continuous, moreover,
there exists a unique real $s(\Phi)$ for which $P(s(\Phi))=0$.
Following~\cite[Chapter~14]{BaranySimonSolomyak_Book23}, we call $s(\Phi)$ the \emph{conformality
dimension} associated to $\Phi$. 

For any self-conformal set and self-conformal measure supported on it, the bounds
\begin{equation}\label{eq:DimUpperBound}
\dim_{\mathrm{H}} \Lambda \leq \min\{1,s(\Phi)\}  \;\;\text{ and }\;\; \dim \mu_{\mathbf{p}} \leq
\min\{1, H(\mathbf{p})/\chi\} \text{ for every } \mathbf{p},
\end{equation} 
hold, regardless of possible overlaps between the pieces $f_i(\Lambda)$.
Here, we denote by $\dim_{\mathrm H}$ the Hausdorff dimension of a set and measure, see
\cite{FalconerBook} for definitions and basic properties.

Rapaport's main result of \cite{Rapaport_SelfConfESC25arXiv}
can be stated in the following way.

\begin{theorem}[\cite{Rapaport_SelfConfESC25arXiv}]\label{thm:RapaportMain}
  Let $\Phi\in\mathfrak{S}_N$ be such that its attractor is not a singleton. If
  $\Phi$ satisfies the ESC, then there is equality in~\cref{eq:DimUpperBound}.
\end{theorem}
We note that the assumption that the attractor is not a singleton is equivalent to the existence of
$f,g\in\Phi$ with distinct fixed points.
The IFSs with non-singleton attractors form an open and dense subset of $\mathfrak{S}_N$.
Combining \cref{thm:ESCOpenDense,thm:RapaportMain} gives the following corollary.

\begin{corollary}\label{cor:dim}
  The set $\{\Phi \colon \Phi \text{ satisfies equality in \cref{eq:DimUpperBound}}\}\subseteq
  \mathfrak{S}_N$ contains an open and dense subset in the $\mathcal{C}^2$ topology.
\end{corollary}
This can be interpreted as follows: a typical analytic self-conformal IFS, in a strong topological
sense, has no dimension drop.

Another potential direction to consider is that of $L^q$ dimensions, a more fine-grained notion that
captures the measure's regularity.  Building on the methods of~\cite{Hochman_SelfSimESC_Annals},
Shmerkin~\cite{Shmerkin_LqSelfSim_Annals} showed that a natural analogue of \cref{thm:RapaportMain}
exists for the $L^q$ dimension of self-similar measures. 
Our typicality result does not extend to the $L^q$ dimension of self-conformal measures in general.
Below, we give an explicit example of an IFS that satisfies the SESC, but the natural measure
exhibits a dimension drop of the $L^q$ dimension for large $q$.
In the setting of M\"obius IFSs satisfying the SESC,
Usuki~\cite{Usuki_LqMobius25arXiv} proved a dichotomy for the $L^q$ spectrum: either
the $L^q$ dimension takes the expected value for all $q>1$, or the spectrum eventually becomes
linear. A key observation is that the latter case occurs whenever two maps in the IFS share a
common fixed point in the attractor, which is precisely the mechanism behind the dimension drop in
our example below.
A natural open question is whether Usuki's result extends to analytic self-conformal
systems.

\subsubsection{Concrete examples}\label{sec:examples}
Similar to the self-similar setting, it would be desirable to have conditions under which one
could determine whether a concrete IFS or a parametrised family of IFSs satisfies the (S)ESC or not
also in the analytic setting. Rapaport~\cite[Corollary 1.4]{Rapaport_SelfConfESC25arXiv}, based on a
result of Solomyak and Takahashi~\cite{SolomyakTakahashi_IMRN21}, showed that under a mild
non-degeneracy condition, given a one-parameter family of analytic IFSs, the set of parameters for
which the ESC fails has zero Hausdorff dimension. Concrete families of IFSs can be constructed to
which this result applies, however, it still does not explicitly say whether an IFS is in the
exceptional set of parameters or not. 
In the following, we give an easy to verify sufficient condition and use it to establish the SESC
for two explicit analytic IFSs.

Given an IFS $\Phi=(f_i)_{i\in\mathcal{I}}\in\mathfrak{S}_N$ and its chosen admissible radius
$\eps$, let
\begin{equation}\label{eq:maxmincontract}
  c_{\min} := \inf_{\substack{x\in\overline{\mathcal{B}_{\eps}} \\ i\in\mathcal{I}}}
  |f_i'(x)| \quad\text{and}\quad
  c_{\max}:=\sup_{\substack{x\in\overline{\mathcal{B}_{\eps}} \\ i\in\mathcal{I}}}
  |f_i'(x)|.
\end{equation}
We note that by \cref{it:c} and compactness, $0<c_{\min} \leq c_{\max} <1$.
We suppress the dependence of these constants on $\Phi$ and $\eps$.

\begin{proposition}\label{prop:example}
Let $\Phi=(f_i)_{i\in\mathcal{I}}\in\mathfrak{S}_N$. Suppose that there exists $\alpha>0$ such that
for each $i\neq j\in\mathcal{I}$ there exists $x_{i,j}\in[0,1]$ with
\begin{equation*}
\left| \frac{f_i''(x_{i,j})}{f_i'(x_{i,j})} - \frac{f_j''(x_{i,j})}{f_j'(x_{i,j})} \right| \geq \alpha . 
\end{equation*}
Let
\begin{equation*}
  \beta:=\sup_{\substack{x\in[0,1] \\ i\in\mathcal{I}}} \left| \frac{f_i''(x)}{f_i'(x)} \right|.
\end{equation*}
Then $\beta>0$ and if $\alpha > 2\beta\cdot c_{\max}/(1-c_{\max})$, we have
$
\sup_{x\in[0,1]} |H_{\bi}(x) - H_{\bj}(x)| > 0 
$
for all distinct $\bi,\bj \in\Sigma\cup\Sigma_*$ with $|\bi|=|\bj|$. In particular, $\Phi$ satisfies
the SESC by~\cref{thm:main}.
\end{proposition}

\begin{proof}
First, let us show that for every $\bi,\bj\in\Sigma\cup\Sigma_*$ with $i_1\neq j_1$ we have for
every non-degenerate closed interval $J\subseteq[0,1]$
\begin{equation}\label{eq:posallint}
	\sup_{x\in J} |H_{\bi}(x) - H_{\bj}(x)| > 0.
\end{equation} 
Since $i_{1}\neq j_{1}$, for the particular choice of $x_{i_{1},j_{1}}$, we can bound 
\begin{align*}
	|H_{\bi}(x_{i_{1},j_{1}}) - H_{\bj}(x_{i_{1},j_{1}})| 
	&\geq \left| \frac{f_{i_{1}}''(x_{i_{1},j_{1}})}{f_{i_{1}}'(x_{i_{1},j_{1}})} -
	\frac{f_{j_{1}}''(x_{i_{1},j_{1}})}{f_{j_{1}}'(x_{i_{1},j_{1}})} \right| - 2
	\sum_{n=2}^{\infty}\; \sup_{x\in[0,1]} \big| f_{\bi_{n-1}^{1}}'(x) \big| \cdot \sup_{x\in[0,1]}
	\left| \frac{f_{i_n}''(x)}{f_{i_n}'(x)} \right| \\
	&\geq \alpha -  2\frac{c_{\max}}{1-c_{\max}}\cdot \beta >0
\end{align*}
by our assumption. Hence, $H_{\bi}\not\equiv H_{\bj}$ and using the analyticity of the maps (which
will be verified later in \cref{thm:H_iAnalytic}) \cref{eq:posallint} follows.
	
Now, let $\bi,\bj \in\Sigma\cup\Sigma_*$ be distinct words such that $|\bi|=|\bj|$ and $|\bi\wedge \bj|=k$. It
  follows from~\cref{eq:H_i(x)} that
  \begin{equation*}
    H_{\bi}(x)-H_{\bj}(x) = f_{\bi_k^1}'(x)\cdot \big( H_{\bi_{k+1}^{|\bi|}}(f_{\bi_k^1}(x)) -
    H_{\bj_{k+1}^{|\bj|}}(f_{\bi_k^1}(x)) \big)
\end{equation*}
for every $x\in[0,1]$. As a result,
\begin{equation*}
|H_{\bi}(x)-H_{\bj}(x)| \geq c_{\min}^k\cdot \big| H_{\bi_{k+1}^{|\bi|}}(f_{\bi_k^1}(x)) -
H_{\bj_{k+1}^{|\bj|}}(f_{\bi_k^1}(x)) \big|.
\end{equation*}
Taking supremum over $x\in[0,1]$ in both sides, the claim of the proposition follows by \cref{eq:posallint}.
\end{proof}

We demonstrate \cref{prop:example} on an example. Consider the IFS $\Phi=(f_1,f_2,f_3)$ with
\begin{equation*}
f_1(x)\coloneqq \frac{x}{8}, \quad f_2(x)\coloneqq \frac{x}{8} + \frac{x^2}{32}, \text{ and } \quad
f_3(x)\coloneqq \frac{x}{16} + \frac{x^2}{32} + \frac{29}{32}. 
\end{equation*}
Cylinder sets certainly overlap heavily since $f_1$ and $f_2$ have common fixed point at $x=0$,
while $f_3$ has fixed point at $x=1$. Since
$\max_{i\in\mathcal{I}}\sup_{x\in I}|f_i'(x)|=3/16$, we may choose an admissible radius
$\eps>0$ sufficiently small that $c_{\max}<1/5$. Moreover,
\begin{equation*}
\sup_{\substack{x\in[0,1] \\ i\in\mathcal{I}}} \left| \frac{f_i''(x)}{f_i'(x)} \right| \leq
\sup_{x\in[0,1]} \max\Big\{\frac{1}{1+x},\, \frac{1}{2+x} \Big\} = 1 \quad\text{ and }\quad
\min_{i\neq j\in\mathcal{I}}\left| \frac{f_i''(0)}{f_i'(0)} - \frac{f_j''(0)}{f_j'(0)} \right| \geq
\frac{1}{2} .
\end{equation*}
Since $c_{\max}<1/5$ implies $2\beta\cdot c_{\max}/(1-c_{\max})<1/2=\alpha$, it follows
from~\cref{prop:example} that $\Phi$ satisfies the SESC, and hence the ESC. The attractor of $\Phi$
is not a singleton, as $f_1$ fixes $0$ and $f_3$ fixes $1$. Moreover,
$\sum_{i\in\mathcal{I}}\sup_{x\in[0,1]}|f_i'(x)|=7/16<1$, so $s(\Phi)<1$, and 
$H(\mathbf{p})<\chi$ for every $\mathbf{p}$. Hence \cref{thm:RapaportMain} gives
$\dim_{\rm H}\Lambda=s(\Phi)$, where $s(\Phi)$ is the conformality dimension of $\Phi$, and
$\dim \mu_{\mathbf{p}} = H(\mathbf{p})/\chi$ for every $\mathbf{p}$, defined in the previous
section. 

Let us note that Solomyak~\cite{Solomyak24} has already given an example of a non-linear conformal
IFS of linear fractional transformations with such a common fixed point structure.
However, he studied the dimension of the attractor via sufficiently large subsystems, and it was not
verified that the IFS itself satisfies the ESC.

We also remark that the $L^q$ dimension of the natural measure in the example
above drops for large $q$. In particular, one can show that the local dimension at $0$ of the
natural measure is strictly smaller than the conformality dimension $s(\Phi)$. However, if the $L^q$
dimension did not drop for every $q>0$ then by \cite[Lemma~1.7]{Shmerkin_LqSelfSim_Annals} the local
dimension at every point would be at least $s(\Phi)$, which is a contradiction. We leave the details
for the interested reader.
The full calculations are included in the accompanying Lean development.

We chose polynomials with algebraic coefficients in the preceding example purely for convenience.
Neither property is required by \cref{prop:example}. For instance, the same argument applies to
the IFS $\Phi=(f_1,f_2,f_3)$ on $I$ given by
\[
  f_1(x)\coloneqq\frac{1-e^{-x/2}}{4},\qquad
  f_2(x)\coloneqq\frac{\sqrt{1+x}-1}{4},\qquad
  f_3(x)\coloneqq\frac{x}{6}+\frac56.
\]
Again, $f_1$ and $f_2$ fix $0$, while $f_3$ fixes $1$. In fact,
$f_1'(0)=f_2'(0)=1/8$ and their pre-Schwarzian derivatives also agree at $0$.
Taking $x_{i,j}=1$ for all $i\neq j$ gives $\alpha=1/4$ and $\beta=1/2$ in
\cref{prop:example}, while $\max_{i\in\mathcal{I}}\sup_{x\in I}|f_i'(x)|=1/6$.
We may therefore choose an admissible radius with $c_{\max}<1/5$, giving
$2\beta c_{\max}/(1-c_{\max})<1/4=\alpha$. Thus this IFS also satisfies the SESC.
The calculations for this example are also included in the accompanying Lean development.


\subsubsection{Conjugation to self-similar systems}\label{sec:ConjLinSys}
As we remarked above, assumption \eqref{eq:H_iSSC} of \cref{thm:main} cannot be satisfied for linear
systems since $H_{\bi}(x)=0$ for all $\bi\in \Sigma$ and $x \in I$.  Crucially, the functions
$H_{\bi}(x)$ have a stronger relation to the linearity and linearisability of an analytic IFS,
although the definition may not at first glance reveal this.  The functions $H_{\bi}$ can be used to
determine whether an analytic IFS can be transformed into a
self-similar one through a change of coordinates.
The problem of being conjugated to a linear IFS played a significant role in the study of Fourier
decay of self-conformal measures, see
\cite{AlgomHertzWang23,AlgomEtal_LogFourierDecaySelfConf_JLMS,BakerBanaji_PolyFourierDecay} and \cite[Corollary 1.2 part
3.]{AlgomEtal_PointwiseNormalityFDecaySelfSonf_AdvMath21}.

\begin{definition}\label{def:conjugation}
  We say that the IFS $\Phi = (f_i)_{i\in\mathcal{I}} \in\mathfrak{S}_N$ is 
\emph{conjugated} to another IFS $\Psi$ if there exists an analytic, invertible
  $g:[0,1]\to\bbR$ such that $\Psi=(g\circ f_i\circ g^{-1})_{i\in\mathcal{I}}$. 
\end{definition}
In particular, $\Phi$
  is conjugated to a self-similar IFS if there exist $\lambda_i
\in(-1,1)\setminus\{0\}$, $t_i\in\bbR$ with $i\in\mathcal{I}$ such that
\[
f_{j}(x) = g^{-1}(\lambda_j g(x) + t_j)
\]
for all $j\in\mathcal{I}$.
\begin{definition}
  We say that the IFS $\Phi = (f_i)_{i\in\mathcal{I}} \in\mathfrak{S}_N$ is 
  \emph{sub-conjugated} to a self-similar IFS if there exist distinct words $\bi,\bj\in\Sigma_*$ of
  the same length such that $(f_{\bi},f_{\bj})$ is conjugated to a self-similar IFS.
\end{definition}

\begin{remark}
  Note that we assume the conjugating function $g$ to be analytic in~\cref{def:conjugation}. One could
  impose the weaker condition that $f\in\mathcal{C}^r([0,1])$ for all $f\in\Phi$ for some $2\leq r\leq
  \infty$ instead. Under this assumption, the authors of~\cite{AlgomEtal_NonLinHyperbolicIFS} show the
  existence of a $\mathcal{C}^r$-smooth IFS which is not $\mathcal{C}^r$-conjugate to self-similar
  even though $f'\equiv c_{\Phi}$ on $\Lambda$ and $f''\equiv 0$ for every $f\in\Phi$. This behaviour
  is not possible in the analytic setting since the assumption that $f''(x)=0$ on $\Lambda$ together
  with analyticity already forces $f$ to be an affine function. 
\end{remark}

We give a characterisation of when an analytic IFS is (sub-)conjugated to a self-similar IFS using
the function $H_{\bi}$ introduced in~\cref{eq:H_i(x)}. We point out that Algom, Rodriguez
Hertz, and Wang \cite[Claim 6.1]{AlgomHertzWang23} already gave a dichotomy stating that 
either (a) there is a analytic conjugation to a self-similar IFS or (b) that any such conjugation
cannot even be $\mathcal{C}^2$. Our result can be considered
complementary, as we give a checkable condition when we are conjugated to a self-similar system.

\begin{theorem}\label{thm:SubConjugation}
  Let $\Phi=(f_i)_{i\in\mathcal{I}}\in\mathfrak{S}_N$. For every $i\in\mathcal{I}$ there exists
  an analytic, invertible map $g\colon[0,1]\to\mathbb{R}$ with $g'(x)\neq0$ for every $x\in[0,1]$
  such that $g\circ f_i\circ g^{-1}$ is a contracting similarity on $g([0,1])$.
  Moreover, if the attractor of $\Phi$ is not a singleton, then
  \begin{enumerate}[label=(\alph*)]
    \item\label{it:conj1} $\Phi$ is conjugated to a self-similar IFS if and only if for every
      $\bi,\bj\in\Sigma$,
      \[
	H_{\bi}(x) \equiv H_{\bj}(x)
      \]
      for all $x\in[0,1]$;
    \item\label{it:conj2} $\Phi$ is sub-conjugated to a self-similar IFS if and only if there exist
      distinct $\bi,\bj\in\Sigma_*$
      of the same length such that 
      \[
	H_{(\bi)^\infty}(x) \equiv H_{(\bj)^\infty}(x)
      \]
      for all $x\in[0,1]$.
  \end{enumerate}
\end{theorem}
\cref{thm:SubConjugation} is proved in~\cref{sec:ProofConjugation}.

A strongly related condition is that of uniform non-integrability (UNI) introduced by
Dolgopyat and Chernov, which was used in the IFS context by Baker and Sahlsten to determine Fourier
decay of measures, see \cite{BakerSahlsten23} and references therein.
In our terminology, the UNI condition is equivalent to the existence of 
$\bi,\bj\in\Sigma$ such that 
$\inf_{x\in I} |H_{\bi}(x) -H_{\bj}(x)| >0$, which is clearly stronger than the negation of
\cref{it:conj1} in \cref{thm:SubConjugation}.
Lequen and Sahlsten~\cite{LequenSahlsten26} also use the dual IFS viewpoint in their work on
quantitative Fourier decay for Patterson--Sullivan measures.

\begin{remark}
  Let us note that in \cref{it:conj1} one only needs to check that $H_{i^\infty}(x)\equiv
  H_{j^\infty}(x)$ for all $i,j \in\mathcal{I}$ and $x\in I$.
  This may not be clear at this point, but will be immediate after the introduction of the dual IFS
  and \cref{thm:DualConj}.
\end{remark}

We return to the discussion of the dimension drop conjecture. Recall
that the ESC implies that the IFS has no exact overlaps. For some time it was an important open
problem whether there exists a self-similar IFS which has no exact overlaps but has super
exponential condensation. Independently of each other, using different methods,
Baker~\cite{Baker_SuperExpCloseIFS_AdvMath} and
B\'ar\'any--K\"{a}enm\"{a}ki~\cite{BaranyKaenmaki_SuperExpCLoseIFS} showed that such examples do
exist. The idea of Baker was further developed in~\cite{Baker_SuperExpCloseIFS2_ProcAMS}
and~\cite{Chen_SuperExpNoExactOVerlapExample}. It follows from the work of
Rapaport~\cite{Rapaport_ExactOverlapsAlgebraicContr} that the examples
in~\cite{Baker_SuperExpCloseIFS_AdvMath, Baker_SuperExpCloseIFS2_ProcAMS,
Chen_SuperExpNoExactOVerlapExample} further support the dimension drop conjecture. 
It is
natural to ask whether analytic IFSs exist which have super-exponential condensation but no exact
overlaps. One may conjecture\footnote{As we have in an earlier version of this article} that any
analytic IFS which has super-exponential condensation but no exact overlaps must be sub-conjugated
to a self-similar IFS. This is however not the case as one can lift self-similar sets with
super-exponential condensation to a M\"obius mapping in a way that makes the system not sub-conjugate
to any self-similar IFS and is the subject of an article by the authors currently in preparation.

We remark that any analytic IFS with a non-singleton attractor and an exact overlap is
sub-conjugated to a self-similar IFS. This follows from analyticity and the first assertion of
\cref{thm:SubConjugation}.

\section{The key idea: the dual IFS induced by analytic functions}\label{sec:DualIFSFull}
The central idea of our work is to construct a `dual' IFS on the space of analytic functions, derived from
the mappings of the original IFS on $[0,1]$.
We believe this notion is of general interest in its own right and a systematic study of it could
assist in tackling other problems in the future as well.

\subsection{Basic definitions and properties}

For every $\eps>0$, let $\mathcal{C}_\eps^\omega([0,1])$ be the set of complex analytic maps
$h\colon\mathcal{B}_\eps\to\mathbb{C}$ which extend continuously to
$\overline{\mathcal{B}_\eps}$ and satisfy $h(I)\subseteq\mathbb{R}$. We equip
$\mathcal{C}_\eps^\omega([0,1])$ with the supremum norm $\|\cdot\|_\infty$ over
$\mathcal{B}_\eps$.

Fix an analytic IFS $\Phi=(f_i)_{i=1}^N\in\mathfrak{S}_N$ and an admissible radius $\eps>0$.
We keep this radius fixed in the construction of the dual IFS, its attractor and its natural
projection. We `lift' each map $f_i$ to an operator $F_i:
\mathcal{C}^{\omega}_\eps([0,1]) \to \mathcal{C}_\eps^{\omega}([0,1])$ acting on the space of analytic
functions by the formula
\begin{equation}\label{eq:LiftedIFS}
	(F_i h)(x)\coloneqq f'_i(x)\cdot h(f_i(x)) + \frac{f''_i}{f'_i}(x).
\end{equation}
We call $\Phi^*\coloneqq(F_i)_{i=1}^N$ the \emph{dual IFS} of $\Phi$. The operator $F_i$ can be
considered as a contractive affinity map and $\Phi^*$ as self-affine IFS on
$\mathcal{C}_\eps^\omega([0,1])$.  Indeed, $F_i$ is a translation of the linear operator
$h\mapsto f'_i(x)\cdot h(f_i(x))$, and each $F_i$ is clearly a strict contraction in the supremum
norm since $\|F_ig-F_ih\|_{\infty}\leq \|f'_i\|_{\infty}\cdot\|g-h\|_{\infty}$, where
$\|f_i'\|_\infty<1$ by our assumption \cref{it:c}.

Our objective now is to establish some basic properties about the dual IFS. We first justify calling
$\Phi^*$ an IFS by showing that it has an attractor. We use the convention that if $A\subset
\mathcal{C}^{\omega}_\eps([0,1])$, then
\begin{equation*}
	F_iA \coloneqq \{F_i h:\, h\in A\}.
\end{equation*}
\begin{lemma}\label{lem:ExistanceAttractor}
	Let $\Phi^*$ be the dual IFS of an analytic IFS $\Phi\in\mathfrak{S}_N$. There exists a
	unique, non-empty, compact set $\Lambda^*\subset \mathcal{C}^{\omega}_\eps([0,1])$,
	which we call the \emph{attractor} of $\Phi^*$, that satisfies
	\begin{equation*}
		\Lambda^*=\bigcup_{i\in\mathcal{I}} F_i\Lambda^*.
	\end{equation*} 
\end{lemma} 

\begin{proof}
For $L>0$, we define 
$$
\mathcal{C}_{\eps,L}^{\omega}([0,1])\coloneqq \big\{g\in\mathcal{C}_\eps^{\omega}([0,1]):\,
|g(x)|\leq L\text{ for every }x\in\mathcal{B}_\eps\big\},
$$
and
$\mathcal{C}_{\eps}^{\omega}([0,1])=\bigcup_{L=1}^\infty\mathcal{C}_{\eps,L}^{\omega}([0,1])$.
It is well-known that the space of continuous and bounded maps is complete with respect to the
supremum distance. By applying Morera's Theorem~\cite[Theorem 10.17]{Rudin_AnalysisBook}, it follows
that $\mathcal{C}_{\eps,L}^{\omega}([0,1])$ is a complete and separable metric space for every
$L>1$.	
	
Since $\|F_ih\|_\infty\leq\|f_i'' / f_i'\|_\infty+\|f_i'\|_\infty\|h\|_\infty$, there exists $L>0$
sufficiently large such that $\|F_ih\|_\infty\leq L$ if $\|h\|_\infty\leq L$. Hence, the claim of
the lemma follows by \cite{Hutchinson_Attractor_81}.
\end{proof}
Let us recall the strong separation condition (SSC) in the context of dual IFS $\Phi^*$. The SSC
holds if $F_i\Lambda^*\cap F_j\Lambda^*=\varnothing$ for every $i\neq j$ in $\mathcal{I}$, that is,
there is no $h\in\Lambda^*$ such that $h\in F_i\Lambda^*$ and $h\in F_j\Lambda^*$.

For $0<\rho\leq\eps$, restriction to $\mathcal{B}_\rho$ intertwines the corresponding dual
operators and maps $\Lambda^*$ onto the attractor in $\mathcal{C}_\rho^\omega([0,1])$.
Restriction is injective by the identity theorem. Thus the SSC, the property that the attractor
is a singleton, and equality of fixed points of dual compositions are independent of the
admissible radius. We suppress this radius in the notation $\Phi^*$ and $\Lambda^*$.

Using the dual IFS $\Phi^*$ and the attractor $\Lambda^*$ of the dual IFS, Theorem~\ref{thm:main}
can be restated in the following elegant form.
\begin{theorem}\label{thm:DualSSC}
	An IFS $\Phi\in\mathfrak{S}_N$ satisfies the SESC if its dual IFS $\Phi^*$ satisfies the SSC.
\end{theorem}

\cref{thm:SubConjugation} also has an equivalent formalisation as follows:

\begin{theorem}\label{thm:DualConj}
	An IFS $\Phi\in\mathfrak{S}_N$ can be conjugated to a self-similar IFS if and only if the
	attractor of its dual IFS $\Phi^*$ is a singleton. 
	
	Moreover, an IFS $\Phi\in\mathfrak{S}_N$ can be sub-conjugated to a self-similar IFS if and
	only if there exist distinct
	$\bi,\bj\in\Sigma_*$ of the same length such that $F_{\bi}$ and $F_{\bj}$ have the same
	fixed point.
\end{theorem}

\begin{remark}
  The same construction applies to conjugation to M\"obius maps by analytic diffeomorphisms,
  using the \emph{Schwarzian derivative} and the operators
  \[
    Sf\coloneqq\frac{f'''}{f'}-\frac32\left(\frac{f''}{f'}\right)^2,
    \qquad
    \widetilde F_i h\coloneqq(f_i')^2(h\circ f_i)+Sf_i.
  \]
  The corresponding dual attractor is a singleton precisely when such a conjugacy exists.
  The proofs are virtually identical to those for conjugation to self-similar iterated function
  systems and are left to the reader.
  K\"aenm\"aki~\cite{Kaenmaki_ExpSepSelfConf26arXiv} has since extended both the pre-Schwarzian and
  the Schwarzian version of this construction to conformal iterated function systems in higher
  dimensions.
\end{remark}

We postpone the explicit proof of \cref{thm:DualSSC} until the end of~\cref{sec:ProofSufficientCond} and
\cref{thm:DualConj} until \cref{sec:ProofConjugation}.

For $\bi\in\Sigma_*$, an induction argument readily gives that the composition
$F_{\bi}h=F_{i_1}\circ\ldots\circ F_{i_{|\bi|}}h$ is equal to
\begin{equation*}
	(F_{\bi}h)(x) = f_{\bi_{|\bi|}^1}'(x)\cdot h(f_{\bi_{|\bi|}^1}(x)) + \sum_{n=1}^{|\bi|}
	f'_{\bi_{n-1}^1}(x) \cdot 
	\frac{f''_{i_n}}{f'_{i_n}}(f_{\bi_{n-1}^1}(x)).
\end{equation*}
Observe that for every $\bi\in\Sigma_*$, taking $h$ equal to the constant $0$ function gives
$(F_{\bi}0)(x)\equiv H_{\bi}(x)$, which was introduced
in~\cref{eq:H_i(x)}. Moreover, for every $\bi\in\Sigma$
\begin{equation*}
	H_{\bi}(x)=\lim_{n\to\infty}(F_{\bi_1^n}h)(x)
\end{equation*}
in the uniform sense for every $h\in\mathcal{C}_\eps^\omega([0,1])$. With this interpretation
the function $H_{\bi}(x)$ is an analog of the natural projection $\pi(\bi)$ from~\cref{eq:natProj},
and this motivates us to call $H_{\bi}(x)$ the \emph{dual natural projection}. The following 
further justifies this nomenclature.
\begin{lemma}\label{thm:H_iAnalytic}
	For every $\bi\in\Sigma$, we have $H_{\bi}\in\mathcal{C}_\eps^{\omega}([0,1])$.
	Moreover, the map $\bi\mapsto H_{\bi}$ is H\"older continuous, that is, there exists
	$K=K(\Phi,\eps)>0$
	such that for every distinct $\bi,\bj\in\Sigma$,
	\begin{equation*}
		\|H_{\bi}-H_{\bj}\|_\infty\leq c_{\max}^{|\bi\wedge\bj|}K.
	\end{equation*}
	In particular, $\Lambda^*=\{H_{\bi}(x):\, \bi\in\Sigma\}$.
\end{lemma}

\begin{proof}
 From the properties in \cref{it:a,it:b,it:c}, it follows that the maps
 \[
 \frac{f_j''}{f_j'},\quad f_{\bi}, \quad\text{and} \quad f'_{\bi}
 \]
 are analytic on a neighbourhood of $\overline{\mathcal{B}_{\eps}}$ for all $\bi\in\Sigma_*$ and $j\in\mathcal{I}$. Hence,
 there exists $C>0$ such that
 \[
   \left|\frac{f_j''}{f_j'}(z)\right| \leq C\text{ and }0<c_{\min}\leq|f_j'(z)|\leq c_{\max}<1
 \]
 for all $j\in\mathcal{I}$ and $z\in\mathcal{B}_{\eps}$, $H_{\bi_1^n}$ converges
 uniformly to $H_{\bi}$ on $\mathcal{B}_{\eps}$ and $H_{\bi}$ is analytic on
 $\mathcal{B}_{\eps}$ by Morera's theorem~\cite[Theorem 10.17]{Rudin_AnalysisBook}. In
 particular, there exists $K'>0$ such that $\|H_{\bi}\|_\infty\leq K'$ for every
 $\bi\in\Sigma$. Hence, $$
 \|H_{\bi}-H_{\bj}\|_\infty\leq c_{\max}^{|\bi\wedge\bj|}2K',
 $$
 which is the claimed bound with $K=2K'$.
 Using the H\"older-continuity, $\{H_{\bi}(x):\, \bi\in\Sigma\}$ is compact, invariant with
 respect to the dual IFS $\Phi^*$, and by the uniqueness of the attractor,
 \cref{lem:ExistanceAttractor}, the last assertion follows.
\end{proof}

We now define cylinder sets for the dual IFS. For two real numbers $a,b$, their convex hull is the interval
$\mathrm{conv}(a,b)=[\min\{a,b\},\max\{a,b\}]$. Slightly abusing notation, a constant
$k\in\mathbb{R}$ also denotes the constant function on $\mathcal{C}_\eps^{\omega}([0,1])$. A
\emph{cylinder set} on $\mathcal{C}_\eps^{\omega}([0,1])$ is given by
$$
(k,K)\coloneqq \big\{g\in\mathcal{C}_\eps^{\omega}([0,1]):\, k<g(x)<K\text{ for every
}x\in[0,1]\big\}.
$$
Since each $F_i$ is a contraction, there exists $k<K$ such that
\begin{equation*}
  k<\min_{i\in\mathcal{I}} \{\min_{x\in[0,1]}(F_ik)(x), \min_{x\in[0,1]}(F_iK)(x)\} \;\text{ and
  }\; K>\max_{i\in\mathcal{I}} \{\max_{x\in[0,1]}(F_ik)(x), \max_{x\in[0,1]}(F_iK)(x)\},
\end{equation*} 
moreover, $F_i(\overline{(k,K)})\subseteq (k,K)$ for every $i\in\mathcal{I}$.
The image of any cylinder $(k,K)$ under $F_{\bi}$ for any
$\bi\in\Sigma_*$ has width
\begin{equation*}
	\max_{x\in[0,1]} |f'_{\bi_{|\bi|}^1}(x)|\cdot (K-k) \leq c_{\max}^{|\bi|} (K-k),
\end{equation*}
where $c_{\max}$ is as in \cref{eq:maxmincontract}. 

We say that two cylinder sets $ F_{\bi}(k,K)$ and $F_{\bj}(k,K)$ are \emph{disjoint}, which we
denote by $F_{\bi}(k,K)\cap F_{\bj}(k,K)=\varnothing$, if there exists $x\in[0,1]$ such that
\begin{equation}\label{eq:DisjointCylinders}
	\mathrm{conv}\big((F_{\bi}k)(x), (F_{\bi}K)(x)\big) \cap \mathrm{conv}\big((F_{\bj}k)(x),
	(F_{\bj}K)(x)\big)= \varnothing.
\end{equation}
If they are not disjoint, we write $F_{\bi}(k,K)\cap F_{\bj}(k,K)\neq\varnothing$.
See~\cref{fig:Cylinders} for an illustration.

\begin{figure}[t]
\includegraphics[width=\linewidth]{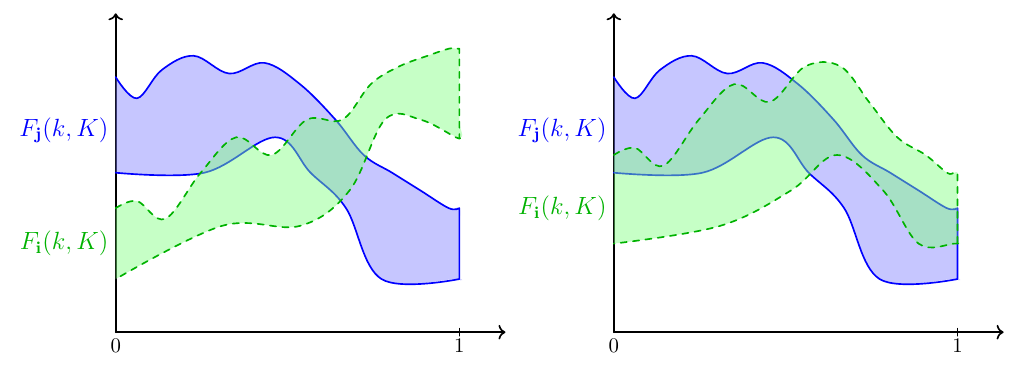}
\caption{Illustration of disjoint cylinders on the left and ones which are not disjoint on the
right.}\label{fig:Cylinders}
\end{figure}

\begin{lemma}\label{lem:SSCEequiv}
The following statements are equivalent: 
\begin{enumerate}[label=(\alph*)]
\item\label{it:la} the dual IFS $\Phi^*$ satisfies the SSC;
\item\label{it:lb} there exist $k<K$ and $n\geq1$ such that $F_{i}\overline{(k,K)}\subseteq(k,K)$ for every
  $i\in\mathcal{I}$ and for every $\bi,\bj\in\Sigma_n$ with $i_1\neq j_1$ we have
\begin{equation*}
	F_{\bi}(k,K)\cap F_{\bj}(k,K)=\varnothing; 
\end{equation*}
\item\label{it:lc} there exists $\delta>0$ such that for all $\bi,\bj \in\Sigma$ with $i_1\neq j_1$ we have
  $\sup_{x\in[0,1]} |H_{\bi}(x) - H_{\bj}(x)| > \delta$.
\end{enumerate}
\end{lemma}

\begin{proof}
  \cref{it:lb}$\Rightarrow$\cref{it:la}: Let $k<K$ be such that
  $F_{i}\overline{(k,K)}\subseteq(k,K)$, and so, $\Lambda^*\subseteq(k,K)$. Thus, the
  implication  is clear. For the other direction, \cref{it:la}$\Rightarrow$\cref{it:lb}, let us
  argue by contradiction. That is, for every $n\geq1$ there exist $\bi,\bj\in\Sigma_n$ with
  $i_1\neq j_1$ such that $F_{\bi}(k,K)\cap F_{\bj}(k,K)\neq\varnothing$. By the compactness
  of $\Sigma$, there exist a subsequence $n_\ell$ and $\bi,\bj\in\Sigma$ with $i_1\neq j_1$
  such that $F_{\bi_1^{n_\ell}}(k,K)\cap F_{\bj_1^{n_\ell}}(k,K)\neq\varnothing$ for every
  $\ell\geq1$. In particular, \cref{eq:DisjointCylinders} implies that for every $x\in[0,1]$
  there exists $y\in\mathbb{R}$ such that $y\in\mathrm{conv}\big((F_{\bi_1^{n_\ell}}k)(x),
    (F_{\bi_1^{n_\ell}}K)(x)\big) \cap \mathrm{conv}\big((F_{\bj_1^{n_\ell}}k)(x),
  (F_{\bj_1^{n_\ell}}K)(x)\big)$. Hence,
  $$
  |F_{\bi_1^{n_\ell}}K(x)-F_{\bj_1^{n_\ell}}K(x)|\leq
  |F_{\bi_1^{n_\ell}}K(x)-y|+|y-F_{\bj_1^{n_\ell}}K(x)|\leq 2c_{\max}^{n_\ell}(K-k).
  $$
  Thus,
  $H_{\bi}=\lim_{\ell\to\infty}F_{\bi_1^{n_\ell}}K=\lim_{\ell\to\infty}F_{\bj_1^{n_\ell}}K=H_{\bj}$
  uniformly, which contradicts to SSC.

  The equivalence \cref{it:la}$\Leftrightarrow$\cref{it:lc} follows by the compactness of $\Lambda^*$,
  \cref{lem:ExistanceAttractor}, and the H\"older continuity of the dual natural projection,
  \cref{thm:H_iAnalytic}. We leave the details for the reader.
\end{proof}

\begin{lemma}\label{lem:cont}
  Let $\Phi^*$ be the dual IFS of an analytic IFS $\Phi\in\mathfrak{S}_N$ such that $\Phi^*$
  satisfies the SSC. Then there exists $\delta>0$ such that for every $\Psi\in\mathfrak{S}_N$
  with $d_2(\Phi,\Psi)<\delta$, the dual IFS $\Psi^*$ of $\Psi$ satisfies the SSC.
\end{lemma}

\begin{proof}
  Let $\Phi=(f_i)_{i=1}^N\in\mathfrak{S}_N$. We compare the restrictions of the dual projections to
  $I$. Let $c_{\min}$ and $c_{\max}$ be as in
  \cref{eq:maxmincontract} for $\Phi$. Moreover, let $C>0$ be such that
  $|f_i''(x)|\leq\left|\frac{f_i''}{f_i'}(x)\right|<C$ and
  $\left|\frac{f_i'''(x)f_i'(x)-(f_i''(x))^2}{(f_i'(x))^2}\right|<C$ for every $i\in\mathcal{I}$ and
  $x\in[0,1]$. Write $c_{\min}'\coloneqq c_{\min}/2$ and $c_{\max}'\coloneqq(1+c_{\max})/2$, giving
  $0<c_{\min}'<c_{\min}\leq c_{\max}<c_{\max}'<1$. Since $|f_i'(x)-g_i'(x)|\leq d_2(\Phi,\Psi)$, the continuity
  of the maps and their derivatives allows us to choose
  $0<\delta\leq\min\{c_{\min}',c_{\max}'-c_{\max}\}$ such that for every
  $\Psi=(g_i)_{i=1}^N\in\mathfrak{S}_N$ with $d_2(\Phi,\Psi)<\delta$,
  $$
  c_{\min}'\leq|g_i'(x)|\leq c_{\max}'\quad\text{ and }\quad
  |g_i''(x)|\leq\left|\frac{g_i''}{g_i'}(x)\right|<C
  $$
  for every $i\in\mathcal{I}$ and $x\in[0,1]$. Moreover, for every $x\in[0,1]$
  \begin{equation}\label{eq:tec1}
    \left|\frac{f_i''}{f_i'}(x)-\frac{g_i''}{g_i'}(x)\right|\leq
    \frac{c_{\max}+C}{(c_{\min}')^{2}}d_2(\Phi,\Psi).
  \end{equation}

  Observe that for every $\bi\in\Sigma_*$ and $x\in[0,1]$
  $$
  |f_{\bi}(x)-g_{\bi}(x)|\leq\sup_{y\in I}|f_{i_1}'(y)|\cdot
  |f_{\bi_2^{|\bi|}}(x)-g_{\bi_2^{|\bi|}}(x)|+\sup_{y\in[0,1]}|f_{i_1}(y)-g_{i_1}(y)|.
  $$ 
  Thus,
  \begin{equation}\label{eq:tec2}
    \sup_{x\in[0,1]}|f_{\bi}(x)-g_{\bi}(x)|\leq \frac{1}{1-c_{\max}}d_2(\Phi,\Psi).
  \end{equation}
  On the other hand, for every $\bi\in\Sigma_*$ and $x\in[0,1]$
  \begin{align*}
    |f_{\bi}'(x)-g_{\bi}'(x)|
    &\leq
    c_{\max}'|f_{\bi_2^{|\bi|}}'(x)-g_{\bi_2^{|\bi|}}'(x)|+
    (c_{\max}')^{|\bi|-1}|f_{i_1}'(f_{\bi_2^{|\bi|}}(x))-g_{i_1}'(g_{\bi_2^{|\bi|}}(x))|\\
    &\leq c_{\max}'|f_{\bi_2^{|\bi|}}'(x)-g_{\bi_2^{|\bi|}}'(x)|+(c_{\max}')^{|\bi|-1}\left(d_2(\Phi,\Psi)+C|f_{\bi_2^{|\bi|}}(x)-g_{\bi_2^{|\bi|}}(x)|\right)\\
    &\leq c_{\max}'|f_{\bi_2^{|\bi|}}'(x)-g_{\bi_2^{|\bi|}}'(x)|+(c_{\max}')^{|\bi|-1}\frac{C+1}{1-c_{\max}}d_2(\Phi,\Psi).
  \end{align*}
  Thus, by induction
  \begin{equation}\label{eq:tec3}
    \sup_{x\in[0,1]}|f_{\bi}'(x)-g_{\bi}'(x)|\leq |\bi|\cdot (c_{\max}')^{|\bi|-1}\frac{C+1}{1-c_{\max}}d_2(\Phi,\Psi).
  \end{equation}

  Combining \cref{eq:tec1,eq:tec2,eq:tec3}, we get that for every $\bi\in\Sigma_*$
  \begin{align*}
&\left|
f'_{\bi_{|\bi|-1}^1}(x) \cdot 
\frac{f''_{i_{|\bi|}}}{f'_{i_{|\bi|}}}(f_{\bi_{|\bi|-1}^1}(x))-
g'_{\bi_{|\bi|-1}^1}(x) \cdot 
\frac{g''_{i_{|\bi|}}}{g'_{i_{|\bi|}}}(g_{\bi_{|\bi|-1}^1}(x))\right|\\
&\leq C |f'_{\bi_{|\bi|-1}^1}(x)-g'_{\bi_{|\bi|-1}^1}(x)|+(c_{\max}')^{|\bi|-1}\left|
\frac{f''_{i_{|\bi|}}}{f'_{i_{|\bi|}}}(f_{\bi_{|\bi|-1}^1}(x))-
\frac{g''_{i_{|\bi|}}}{g'_{i_{|\bi|}}}(g_{\bi_{|\bi|-1}^1}(x))\right|\\
&\leq C |f'_{\bi_{|\bi|-1}^1}(x)-g'_{\bi_{|\bi|-1}^1}(x)|\\
&\qquad+(c_{\max}')^{|\bi|-1}\left(C\left|f_{\bi_{|\bi|-1}^1}(x)-
g_{\bi_{|\bi|-1}^1}(x)\right|+\left|\frac{f_{i}''}{f_i'}(g_{\bi_{|\bi|-1}^1}(x))-\frac{g_{i}''}{g_i'}(g_{\bi_{|\bi|-1}^1}(x))\right|\right)\\
&\leq\left(\frac{|\bi|\cdot C(C+1)}{c_{\max}'(1-c_{\max})}+\frac{C}{1-c_{\max}}+\frac{c_{\max}+C}{(c_{\min}')^2}\right)(c_{\max}')^{|\bi|-1}d_2(\Phi,\Psi).
  \end{align*}
  In particular, there exists $C'>0$ depending on $\Phi$ such that for every $x\in I$ and
  $\bi\in\Sigma\cup\Sigma_*$,
\begin{equation}\label{eq:conteq}
|H_{\bi}(x)-\hat{H}_{\bi}(x)|\leq C'd_2(\Phi,\Psi),
\end{equation}
where $H_{\bi}$ denotes the dual projection of $\Phi^*$ and $\hat{H}_{\bi}$ denotes the dual projection of $\Psi^*$.

Now, if the dual IFS $\Phi^*$ of $\Phi=(f_i)_{i=1}^N\in\mathfrak{S}_N$ satisfies the SSC then by
\cref{lem:SSCEequiv} there exists $\eta>0$ such that
$\sup_{x\in[0,1]}|H_{\bi}(x)-H_{\bj}(x)|\geq\eta$ for every $\bi,\bj\in\Sigma$ with
$i_1\neq j_1$. Hence, by \cref{eq:conteq} and \cref{lem:SSCEequiv}, for every $\Psi$ with
$d_2(\Phi,\Psi)<\min\{\delta,\eta/(3C')\}$ the dual $\Psi^*$ satisfies the SSC.
\end{proof}

\subsection{Further analysis of the dual natural projection \texorpdfstring{$H_{\bi}$}{H}}

After establishing that $H_{\bi}$ is analytic, we wish to obtain bounds on its derivatives.
In this subsection we work on $I$ and write $f^{(k)}$ for the $k$-th derivative of $f$.
For any $\bi\in\Sigma_*$, using the chain rule we get $f_{\bi_{|\bi|}^1}'(x)=
\prod_{n=1}^{|\bi|}f_{i_n}'(f_{\bi_{n-1}^1}(x))$. From here, a simple calculation yields 
that for every finite word $\bi\in\Sigma_*$, $H_{\bi}$ reduces to
\begin{equation}\label{eq:H=f''/f'}
  H_{\bi}(x)=H_{\bi_1^{|\bi|}}(x) = \frac{f''_{\bi_{|\bi|}^1}(x)}{f'_{\bi_{|\bi|}^1}(x)}.
\end{equation}
Another way of writing $H_{\bi}$ for $\bi\in\Sigma\cup\Sigma_*$ is
\begin{equation}
  H_{\bi}(x) = \sum_{n=1}^{|\bi|} (\phi_{i_n}\circ f_{\bi_{n-1}^1})'(x),
  \label{eq:alternativeH}
\end{equation}
where $\phi_{i_k}(x)\coloneqq \log|f'_{i_k}(x)|$. Recall that the $k$-th derivative of the
composition of two functions can be calculated using Fa\`a di Bruno's formula:
\begin{equation}\label{eq:FaaDiBruno}
  (f\circ g)^{(k)}(x) = \sum_{\pi\in \Pi_k} f^{(|\pi|)}(g(x)) \cdot \prod_{B\in\pi}
  g^{(|B|)}(x),
\end{equation}
where $\Pi_k$ is the set of all partitions of $\{1,\dots,k\}$, and $B\in \pi$ refers to the
elements, or blocks, of the partition $\pi$. Finally, let $g_0\coloneqq 1$ and, for $k\geq 1$, let
$g_k:\bbR^k\to \bbR$ be a $k$-variable polynomial such that $g_1(y_1)=y_1$ and the next one is
given by the formula
\[
g_{k+1}(y_{k+1},\dots,y_1)\coloneqq \sum_{\ell=1}^k \frac{\partial g_k}{\partial
y_{\ell}}(y_k,\dots,y_1)\cdot y_{\ell+1} +
g_k(y_k,\dots,y_1)\cdot y_1.
\]
In particular, $g_2(y_2,y_1)=y_2+y_1^2$, $g_3(y_3,y_2,y_1) = y_3+3y_1y_2+y_1^3$ and so on. We are
now ready to give a formula for $H_{\bi}^{(k)}$, the $k$th derivatives of $H_{\bi}$.

\begin{lemma}
\label{thm:ugly}
For any $\bi\in\Sigma\cup\Sigma_*$, the $k$-th derivative of $H_{\bi}$ is given by
\[
H_{\bi}^{(k)}(x)
=\sum_{\pi\in\Pi_{k+1}} \sum_{n=1}^{|\bi|} \phi_{i_n}^{(|\pi|)}(f_{\bi_{n-1}^1}(x))\cdot
(f'_{\bi_{n-1}^1}(x))^{|\pi|}\cdot \prod_{B\in\pi}
g_{|B|-1}\big(H_{\bi_1^{n-1}}^{(|B|-2)}(x),\dots,H_{\bi_1^{n-1}}(x)\big).
\]
\end{lemma}
\begin{proof}
We first show by induction that for any $\bi\in\Sigma_*$,
\begin{equation}
	\label{eq:fandg}
	\frac{f_{\bi_{|\bi|}^1}^{(k)}(x)}{f_{\bi_{|\bi|}^1}'(x)}
	=
	g_{k-1}\big(H_{\bi}^{(k-2)}(x), \dots, H_{\bi}(x)\big).
\end{equation}
Indeed, for $k=2$, \cref{eq:fandg} is the same as~\cref{eq:H=f''/f'}. Differentiating both sides
of~\cref{eq:fandg} we get
\[
\frac{f_{\bi_{|\bi|}^1}^{(k+1)}(x)}{f_{\bi_{|\bi|}^1}'(x)} -
\frac{f_{\bi_{|\bi|}^1}^{(k)}(x)}{f_{\bi_{|\bi|}^1}'(x)}\cdot
\frac{f_{\bi_{|\bi|}^1}''(x)}{f_{\bi_{|\bi|}^1}'(x)}
=\sum_{\ell=0}^{k-2}\frac{\partial g_{k-1}}{\partial
y_{\ell+1}}\big(H_{\bi}^{(k-2)},\dots,H_{\bi}(x)\big) \cdot H_{\bi}^{(\ell+1)}(x).
\]
We use the induction hypothesis~\cref{eq:fandg} in the second term on the left hand side for $k$
and \cref{eq:H=f''/f'} to see that
\[
\frac{f_{\bi_{|\bi|}^1}^{(k)}(x)}{f_{\bi_{|\bi|}^1}'(x)}\cdot
\frac{f_{\bi_{|\bi|}^1}''(x)}{f_{\bi_{|\bi|}^1}'(x)}
=g_{k-1}(H_{\bi}^{(k-2)}(x),\dots, H_{\bi}(x)) \cdot H_{\bi}(x).
\]
Substituting this back, after rearranging the inductive step is proved for $k+1$: 
\begin{align*}
\frac{f_{\bi_{|\bi|}^1}^{(k+1)}(x)}{f_{\bi_{|\bi|}^1}'(x)}
&=
g_{k-1}(H_{\bi}^{(k-2)}(x),\dots, H_{\bi}(x)) \cdot H_{\bi}(x)
+
\sum_{\ell=0}^{k-2}\frac{\partial g_{k-1}}{\partial
y_{\ell+1}}\big(H_{\bi}^{(k-2)},\dots,H_{\bi}(x)\big) \cdot H_{\bi}^{(\ell+1)}(x)
\\
&=
g_{k}(H_{\bi}^{(k-1)}(x), \dots, H_{\bi}(x)).
\end{align*}
We can now derive the formula for $H_{\bi}^{(k)}$:
\begin{align*}
H_{\bi}^{(k)}(x)
&\stackrel{\eqref{eq:alternativeH}}{=}\sum_{n=1}^{|\bi|} (\phi_{i_n}\circ f_{\bi_{n-1}^1})^{(k+1)}(x) \\
&\stackrel{\eqref{eq:FaaDiBruno}}{=}\sum_{\pi\in\Pi_{k+1}} \sum_{n=1}^{|\bi|}
\phi_{i_n}^{(|\pi|)}(f_{\bi_{n-1}^1}(x))\cdot
\prod_{B\in\pi} f_{\bi_{n-1}^1}^{(|B|)}(x) \\
&\stackrel{\eqref{eq:fandg}}{=}\sum_{\pi\in\Pi_{k+1}} \sum_{n=1}^{|\bi|}
\phi_{i_n}^{(|\pi|)}(f_{\bi_{n-1}^1}(x))\cdot
(f'_{\bi_{n-1}^1}(x))^{|\pi|} \cdot \prod_{B\in\pi}
g_{|B|-1}\big(H_{\bi_1^{n-1}}^{(|B|-2)}(x),\dots,H_{\bi_1^{n-1}}(x)\big).
\end{align*}
\end{proof}

\begin{lemma}\label{thm:kbound}
For every integer $k\geq 0$ there exists $C_k$ such that for all $x\in I$ and all
$\bi\in\Sigma\cup\Sigma_*$,
\[
\big|H_{\bi}^{(k)}(x)\big|\leq C_k.
\]
\end{lemma}
\begin{proof}
We proceed by induction. Let
\[
D_k:=\max_{i\in\Sigma_1}\max_{x\in[0,1]} |\phi_i^{(k)}(x)|.
\]
From~\cref{eq:alternativeH} we see that
$|H_{\bi}(x)| \leq D_1/(1-c_{\max})=:C_0$. Set $E_0=1$.
Suppose that the statement is true for $0,\dots,k-1$ and define
\begin{equation}\label{eq:E_k}
E_k:= \sup_{\substack{y_{j+1}\in[-C_j,C_j]\\0\,\leq\, j \,\leq\, k-1}}|g_k(y_k,\dots,y_1)|.
\end{equation}
By \cref{thm:ugly},
\[
\big|H_{\bi}^{(k)}(x)\big| \leq \sum_{\pi\in\Pi_{k+1}}\sum_{n=1}^{|\bi|} D_{|\pi|} \cdot
c_{\max}^{(n-1)|\pi|} \cdot \prod_{B\in\pi}
E_{|B|-1}
\leq \sum_{\pi\in\Pi_{k+1}} \frac{D_{|\pi|}\cdot \prod_{B\in\pi}E_{|B|-1}}{1-c_{\max}^{|\pi|}}
\]
and the statement follows.
\end{proof}
\begin{corollary}
\label{thm:difcor}
For all $k\geq 1$, for all $x,y\in I$ and for all $\bi\in\Sigma\cup\Sigma_*$,
\[
\big|H_{\bi}^{(k)}(x) - H_{\bi}^{(k)}(y)\big| \leq C_{k+1}\cdot |x-y|,
\]
where $C_k>0$ are as defined in \cref{thm:kbound}. 
\end{corollary}
We show the following useful H\"older type bound. 
\begin{lemma}\label{thm:difbound}
For all integers $k\geq 0$, $x\in I$ and distinct $\bi,\bj\in\Sigma\cup\Sigma_*$,
\[
\big|H_{\bi}^{(k)}(x) - H_{\bj}^{(k)}(x)\big| \leq 2C_k\cdot c_{\max}^{|\bi\wedge\bj|},
\]
where the $C_k>0$ are as defined in \cref{thm:kbound}. 
\end{lemma}
\begin{proof}
  Let $m = |\bi\wedge\bj|$. By \cref{thm:ugly,thm:kbound},
  \begin{align*}
	&|H_{\bi}^{(k)}(x) - H_{\bj}^{(k)}(x)|
	\\
	&=
	\left|
	\sum_{\pi\in\Pi_{k+1}} \sum_{n=m+1}^{|\bi|} \phi_{i_n}^{(|\pi|)}(f_{\bi_{n-1}^1}(x))\cdot 
	(f'_{\bi_{n-1}^1}(x))^{|\pi|} \cdot
	\prod_{B\in\pi} g_{|B|-1}\big(H_{\bi_1^{n-1}}^{(|B|-2)}(x),\dots,H_{\bi_1^{n-1}}(x)\big)
      \right.
      \\
	&
	\left.\;\; -
	  \sum_{\pi\in\Pi_{k+1}} \sum_{n=m+1}^{|\bj|} \phi_{j_n}^{(|\pi|)}(f_{\bj_{n-1}^1}(x))\cdot
	  (f'_{\bj_{n-1}^1}(x))^{|\pi|} \cdot
	  \prod_{B\in\pi} g_{|B|-1}\big(H_{\bj_1^{n-1}}^{(|B|-2)}(x),\dots,H_{\bj_1^{n-1}}(x)\big)
	  \right|\\
	&\leq
	\sum_{\pi\in\Pi_{k+1}}\sum_{n=m+1}^\infty 2D_{|\pi|}\cdot c_{\max}^{(n-1)|\pi|} \cdot
	\prod_{B\in\pi} E_{|B|-1}
	\leq 2 C_{k}\cdot c_{\max}^m.
	\qedhere
  \end{align*}
\end{proof}

\section{Proof of the sufficient condition for SESC}\label{sec:ProofSufficientCond}

We first prove an auxiliary lemma, then prove \cref{thm:DualSSC} and deduce \cref{thm:main}.

\begin{lemma}\label{thm:analyticity}
Let $f$ and $g$ be real analytic maps on $J$
and let $\eta>0$ with $2\sqrt{\eta}<|J|$.
Denote
 \[
   Q \coloneqq \max\left\{\sup_{x\in J} |f''(x)|,\, \sup_{x\in J}|g''(x)|\right\}.
 \]
If $\sup_{x\in J} |f(x)-g(x)| \leq \eta$, then $\sup_{x\in J} |f'(x)-g'(x)|\leq (2+Q)\sqrt{\eta}$.
\end{lemma}
\begin{proof}
Let $x\in J$ be arbitrary and take $y\in J$ such that $|x-y|=\sqrt{\eta}$. By assumption
$\max\{|f(x)-g(x)|,|f(y)-g(y)|\}\leq \eta$. Using the second order Taylor approximation 
\[
  f(y) = f(x) + f'(x)(y-x)+ \frac{f''(\xi_1)}{2}(y-x)^2,
\]
where $\xi_1\in(x,y)$, and similarly for $g(y)$ around $x$ we get
\begin{align*}
  \eta &\geq |f(y)-g(y)| =
  \left|f(x)-g(x)+(f'(x)-g'(x))(y-x)+(f''(\xi_1)-g''(\xi_2))\frac{(y-x)^2}{2}\right|\\
 &\geq |f'(x)-g'(x)|\cdot|y-x|-|f(x)-g(x)|-(|f''(\xi_1)|+|g''(\xi_2)|)\frac{(y-x)^2}{2}.
\end{align*}
  Thus,
  \[
    |f'(x)-g'(x)| \leq \frac{\eta+\eta+Q \eta}{\sqrt{\eta}} = (2+Q)\sqrt{\eta}
  \]
  as required.
\end{proof}

\begin{proof}[Proof of \cref{thm:DualSSC}]
Suppose that $\Phi^*$ satisfies the SSC but $\Phi$ does not satisfy the SESC.
By \cref{def:SESC}, there exist a sequence $(\eta_n)_n$ such that $\log(\eta_n)/n\to-\infty$, a
subsequence $n_{\ell}\in\bbN$ and $\bi\neq\bj\in\Sigma_{n_\ell}$ such that
\begin{equation}\label{eq:ExpCondensation}
  \sup_{x\in[0,1]}|f_{\bi}(x)-f_{\bj}(x)| \leq \eta_{n_\ell}.
\end{equation}
For the remainder of the proof we work with the sequence $\eta_{n_{\ell}}$ and $\bi\neq\bj\in\Sigma_{n_\ell}$ provided
by~\cref{eq:ExpCondensation}. Let $m=m(n_{\ell})\coloneqq\max\{k\leq n_{\ell}:\, i_k\neq j_k\}$ and
$\bu^{(n_{\ell})}\coloneqq \bi_{m+1}^{n_{\ell}}\in\Sigma_{n_\ell-m}$, then
$\bj=\bj_1^m\bu^{(n_{\ell})}$. Let us also denote $\bi^{(n_{\ell})}\coloneqq\bi_m^1\in\Sigma_m$ and
$\bj^{(n_{\ell})}\coloneqq\bj_m^1\in\Sigma_m$, so $(\bi^{(n_{\ell})})_1\neq (\bj^{(n_{\ell})})_1$.
We first show that there exists a sequence $\eta_{n_{\ell}}''$ with
$\log(\eta_{n_{\ell}}'')/n_{\ell}\to-\infty$ such that for every $x\in I$,
\begin{equation}\label{eq:DiffH_iSUperExp}
|H_{\bi^{(n_\ell)}}(f_{\bu^{(n_\ell)}}(x)) - H_{\bj^{(n_\ell)}}(f_{\bu^{(n_\ell)}}(x))| \leq \eta''_{n_\ell}.
\end{equation}

For any $\bi\in\Sigma_*$, recall from~\cref{eq:H=f''/f'} that
\begin{equation*}
	H_{\bi}(x)= \frac{f''_{\bi_{|\bi|}^1}(x)}{f'_{\bi_{|\bi|}^1}(x)} \;\;\text{ or equivalently,
	}\;\; H_{\bi_{|\bi|}^1}(x) = \frac{f''_{\bi}(x)}{f'_{\bi}(x)}.
\end{equation*}
Using~\eqref{eq:H=f''/f'}, we get 
\begin{equation*}
 |f_{\bi}''(x)| = \big|H_{\bi_{|\bi|}^1}(x)\cdot f_{\bi}'(x)\big|\leq C_0 \cdot c_{\max}^{|\bi|}
\end{equation*}
by \cref{thm:kbound}. Moreover,
\begin{equation*}
  |f_{\bi}'''(x)| = \big|g_2(H_{\bi_{|\bi|}^1}'(x),H_{\bi_{|\bi|}^1}(x))\cdot f_{\bi}'(x)\big| \leq
  E_2\cdot c_{\max}^{|\bi|}
\end{equation*}
by the definition of $E_k$ in~\cref{eq:E_k}. Since $\eta_{n_\ell}\to 0$ super-exponentially, we may
assume $(2+C_0)\sqrt{\eta_{n_\ell}}<1/4$ for all $\ell$. Together with~\cref{thm:analyticity} these bounds imply
that for the particular choice of $\bi,\bj$ in~\cref{eq:ExpCondensation} we have
\begin{align*}
 \sup_{x\in[0,1]}|f_{\bi}'(x) - f_{\bj}'(x)| &\leq (2+C_0 c_{\max}^{n_\ell})\sqrt{\eta_{n_\ell}}; \\
 \sup_{x\in[0,1]}|f_{\bi}''(x) - f_{\bj}''(x)| &\leq (2+E_2c_{\max}^{n_\ell})\sqrt{2+C_0
 	c_{\max}^{n_\ell}}\cdot
 \eta_{n_\ell}^{1/4}.
\end{align*}
We deduce
\begin{align*}
	\left|\frac{f_{\bi}''(x)}{f_{\bi}'(x)} - \frac{f_{\bj}''(x)}{f_{\bj}'(x)}\right|
	&\leq
	\frac{|f_{\bi}''(x)|}{|f_{\bi}'(x)||f_{\bj}'(x)|}\cdot|f_{\bi}'(x) - f_{\bj}'(x)|
	+\frac{1}{|f'_{\bj}(x)|} \cdot |f_{\bi}''(x) - f_{\bj}''(x)|
	\\
	&
	\leq \frac{C_0}{c_{\min}^{n_\ell}}(2+C_0c_{\max}^{n_\ell})\cdot \eta_{n_\ell}^{1/2}
	+\frac{1}{c_{\min}^{n_\ell}}(2+E_2
	c_{\max}^{n_\ell})\sqrt{2+C_0 c_{\max}^{n_\ell}} \cdot \eta_{n_\ell}^{1/4} \;=:\;\eta_{n_\ell}'.
\end{align*}
Now observe that
\begin{equation*}
H_{\bi_{n_{\ell}}^1}(x) = \frac{(f_{\bi_1^m}\circ f_{\bu^{(n_{\ell})}})''(x)}{(f_{\bi_1^m}\circ
f_{\bu^{(n_{\ell})}})'(x)} = 
f'_{\bu^{(n_{\ell})}}(x) \cdot H_{\bi^{(n_{\ell})}} \big(f_{\bu^{(n_{\ell})}}(x)\big) + \frac{
f''_{\bu^{(n_{\ell})}}(x) }{ f'_{\bu^{(n_{\ell})}}(x)},
\end{equation*}
hence, $H_{\bi_{n_{\ell}}^1}(x)-H_{\bj_{n_{\ell}}^1}(x) = f'_{\bu^{(n_{\ell})}}(x) \cdot \big(
H_{\bi^{(n_{\ell})}} \big(f_{\bu^{(n_{\ell})}}(x)\big) -H_{\bj^{(n_{\ell})}}
\big(f_{\bu^{(n_{\ell})}}(x)\big) \big)$, so we can conclude
\[
\big| H_{\bi^{(n_{\ell})}} \big(f_{\bu^{(n_{\ell})}}(x)\big) -H_{\bj^{(n_{\ell})}}
\big(f_{\bu^{(n_{\ell})}}(x)\big) \big|
 \leq
 \frac{\big|H_{\bi_{n_{\ell}}^1}(x)-H_{\bj_{n_{\ell}}^1}(x)\big|}{\big|f'_{\bu^{(n_{\ell})}}(x)\big|}
 \leq c_{\min}^{-{n_\ell}}
\cdot \eta_{n_\ell}'
=:\eta_{n_\ell}''.
\]

Having established~\cref{eq:DiffH_iSUperExp}, there are two cases to consider: whether
$|\bu^{(n_\ell)}| \to \infty$ or there exists a constant $C>0$ and infinitely many $\ell$ such that
$|\bu^{(n_\ell)}| \leq C$. Let us first assume the latter. Since $|\bi^{(n_\ell)}|+ |\bu^{(n_\ell)}|
= |\bj^{(n_\ell)}|+|\bu^{(n_\ell)}| = n_\ell$ we conclude, by compactness, that
there exists a subsequence $n_\ell'$ such that $\bi^{(n_\ell')} \to \bi^*\in\Sigma$,
$\bj^{(n_\ell')}\to\bj^*\in\Sigma$ and $\bu^{(n_\ell')}=\bu^*\in\Sigma_*$ with $i_1^*\neq j_1^*$.
It follows from~\cref{thm:difbound} and~\cref{eq:DiffH_iSUperExp} that
$
  H_{\bi^*}(f_{\bu^*}(x))\equiv H_{\bj^*}(f_{\bu^*}(x))
$
for all $x\in I$.  Hence, using the analyticity of $H_{\bi}$ from~\cref{thm:H_iAnalytic} we conclude
that $H_{\bi^*}(x) \equiv H_{\bj^*}(x)$ for all $x\in I$, which contradicts the dual SSC by
\cref{lem:SSCEequiv} since $i_1^*\neq j_1^*$.

Now let us assume that $|\bu^{(n_\ell)}| \to \infty$. Again by compactness, there exists
$\bu^*\in\Sigma$ as well as $\bi^*,\bj^*\in\Sigma\cup\Sigma_*$  and a
subsequence $n_\ell'$ such that $\bi^{(n_\ell')}\to \bi^*$ and $\bj^{(n_\ell')}\to \bj^*$ with
$i_1^*\neq j_1^*$ as well as $f_{\bu^{(n_\ell')}}(x) \to \pi(\bu^*)$ for all $x\in[0,1]$. Note that
both $|\bi^*|$ and
$|\bj^*|$ might be finite or infinite, however, $|\bi^*|=|\bj^*|$ by the construction.

Suppose first that $\bi^*$ and $\bj^*$ are finite, so $\bi^{(n_\ell')}=\bi^*$ and
$\bj^{(n_\ell')}=\bj^*$ for all large $\ell$. Since
$|f_{\bu^{(n_\ell')}}([0,1])|\geq c_{\min}^{n_\ell'}$, the super-exponential bound in
\cref{eq:ExpCondensation} forces $f_{(\bi^*)_{|\bi^*|}^1}\equiv f_{(\bj^*)_{|\bj^*|}^1}$ by
finite order of vanishing of a nonzero analytic function at $\pi(\bu^*)$. Hence
$F_{\bi^*}=F_{\bj^*}$, contradicting the dual SSC since $i_1^*\neq j_1^*$.

Suppose now that $\bi^*$ and $\bj^*$ are infinite. For each fixed $k$, we apply
\cref{thm:analyticity} successively to $H_{\bi^{(n_\ell)}}$, $H_{\bj^{(n_\ell)}}$ and their
derivatives on $f_{\bu^{(n_\ell)}}(I)$, using \cref{thm:kbound,eq:DiffH_iSUperExp}.
Since $|f_{\bu^{(n_\ell)}}(I)|\geq c_{\min}^{n_\ell}$ and $\eta_{n_\ell}''$ tends to zero
super-exponentially, the smallness conditions hold for all sufficiently large $\ell$.
Enlarging the constant to cover the finitely many remaining indices and using
$|f'_{\bu^{(n_\ell)}}(x)|\leq1$, we obtain $\widetilde{C}_k>0$ such that for all $\ell\geq1$,
\begin{equation}\label{eq:H_i^(k)Diff}
  \big|H_{\bi^{(n_\ell)}}^{(k)}(f_{\bu^{(n_\ell)}}(x)) - H_{\bj^{(n_\ell)}}^{(k)}(f_{\bu^{(n_\ell)}}(x))\big|
  \leq \widetilde{C}_k \cdot
  \frac{1}{|f'_{\bu^{(n_\ell)}}(x)|^k}\cdot\left(\eta_{n_\ell}''\right)^{2^{-k}}\leq\widetilde{C}_k
  \cdot \frac{1}{c_{\min}^{k n_\ell}}\cdot\left(\eta_{n_\ell}''\right)^{2^{-k}}
\end{equation}
for all $x\in[0,1]$ which still tends to $0$ as $\ell\to\infty$ since $\eta_{n_{\ell}}''\to 0$
super-exponentially fast. Combining~\cref{thm:difcor} and~\cref{thm:difbound}
with~\cref{eq:H_i^(k)Diff} we see that $H_{\bi^*}^{(k)}(\pi(\bu^*))=H_{\bj^*}^{(k)}(\pi(\bu^*))$ for
all $k$.
Since $H_{\bi^*}$ and $H_{\bj^*}$ are analytic by~\cref{thm:H_iAnalytic}, we get
$H_{\bi^*}(x)\equiv H_{\bj^*}(x)$ for all $x\in[0,1]$, which again contradicts the dual SSC by
\cref{lem:SSCEequiv}.
\end{proof}

\begin{proof}[Proof of~\cref{thm:main}]
By compactness and \cref{thm:H_iAnalytic}, condition \cref{eq:H_iSSC} gives the uniform separation
in \cref{it:lc} of \cref{lem:SSCEequiv}. Hence $\Phi^*$ satisfies the SSC, and $\Phi$ satisfies
the SESC by \cref{thm:DualSSC}.
\end{proof}

\section{Existence of open and dense set of IFSs with SESC}\label{sec:ProofESCOpenDense}

\subsection{Preliminaries}

Fix an arbitrary $n\geq 1$. Let $\mathcal{B}_n\coloneqq\{(\bi,\bj)\in\Sigma_n\times\Sigma_n:\,
i_1<j_1\}$. We say that $(\bi,\bj)\in\mathcal{B}_n$ is bad if the associated cylinders overlap,
\begin{equation}\label{eq:DefBad}
F_{\bi}(k,K)\cap F_{\bj}(k,K)\neq\varnothing,
\end{equation}
recall~\cref{eq:DisjointCylinders}. For any $\bi\in\Sigma_*$ and $x\in[0,1]$ define the orbit of
$\bi$ starting from $x$ as the multiset $\mathcal{O}_{\bi}(x)\coloneqq \{x,f_{i_1}(x),
f_{\bi_2^1}(x),\ldots,f_{\bi_{|\bi|}^1}(x)\}$.

\begin{lemma}\label{lem:Pointsx_ij}
Let $(f_i)_{i\in\mathcal{I}}\in\mathfrak{S}_N$ be such that $f_{\bi}\not\equiv f_{\bj}$ on $[0,1]$ for
all distinct $\bi,\bj\in\Sigma_*$, which holds in particular if the IFS has no exact overlaps. Then
there exists a collection of points
$\{x_{\bi,\bj}\}_{(\bi,\bj)\in\mathcal{B}_n}\subseteq[0,1]$ such that
\begin{enumerate}[label=(\alph*)]
  \item\label{it:lla} $\mathrm{conv}\big((F_{\bi}k)(x_{\bi,\bj}), (F_{\bi}K)(x_{\bi,\bj})\big) \cap
  \mathrm{conv}\big((F_{\bj}k)(x_{\bi,\bj}),
	(F_{\bj}K)(x_{\bi,\bj})\big)= \varnothing$ if $(\bi,\bj)$ is not bad;  
      \item\label{it:llb} all points in $\mathcal{O}_{\bi}(x_{\bi,\bj})$ and also in $\mathcal{O}_{\bj}(x_{\bi,\bj})$
  are distinct, moreover,
  $\mathcal{O}_{\bi}(x_{\bi,\bj})\cap\mathcal{O}_{\bj}(x_{\bi,\bj})=\{x_{\bi,\bj}\}$;
\item\label{it:llc} $\big(\mathcal{O}_{\bi}(x_{\bi,\bj})\cup \mathcal{O}_{\bj}(x_{\bi,\bj})\big) \cap
  \big(\mathcal{O}_{\bh}(x_{\bh,\bk})\cup
  \mathcal{O}_{\bk}(x_{\bh,\bk})\big)=\varnothing$ for every $(\bi,\bj)\neq
  (\bh,\bk)\in\mathcal{B}_n$.
\end{enumerate} 
\end{lemma}

\begin{proof}
We first establish an order on the elements of $\mathcal{B}_n$ by setting
\[
  (\bi^{(r)},\bj^{(r)})=(i_1^{(r)},\ldots,
i_n^{(r)};j_1^{(r)},\ldots,j_n^{(r)}),
\]
where $r=1,\ldots,\#\mathcal{B}_n$. The points are constructed
inductively. If $(\bi^{(1)},\bj^{(1)})$ is not bad then by definition there exists an $\hat
x_{(\bi^{(1)},\bj^{(1)})}$ for which \cref{it:lla} holds. Since \cref{it:lla} is an open condition, there exists an
$x_{(\bi^{(1)},\bj^{(1)})}$ in the neighborhood of $\hat x_{(\bi^{(1)},\bj^{(1)})}$ for which \cref{it:llb}
also holds. If this were not the case, then there would be $k,\ell$ such that
$f_{(\bi^{(1)})_{\ell}^1}(x)=f_{(\bi^{(1)})_{k}^1}(x)$ for infinitely many $x$, but then analyticity
implies that $f_{(\bi^{(1)})_{\ell}^1}(x)\equiv f_{(\bi^{(1)})_{k}^1}(x)$ on $[0,1]$, which contradicts our
assumption. If $(\bi^{(1)},\bj^{(1)})$ is bad, then choose
$x_{(\bi^{(1)},\bj^{(1)})}$ to satisfy \cref{it:llb} (which is possible by a similar argument). Thus we have
constructed the first point $x_{(\bi^{(1)},\bj^{(1)})}$. Condition \cref{it:llc} trivially holds with just
the single pair $(\bi^{(1)},\bj^{(1)})$.

We continue by induction. Assume that $x_{(\bi^{(1)},\bj^{(1)})},\ldots,x_{(\bi^{(r)},\bj^{(r)})}$
have already been constructed (for some $r\geq 1$) so that \cref{it:lla,it:llb,it:llc} all hold. The set
$\bigcup_{m=1}^r \widehat{\mathcal{O}}(x_{(\bi^{(m)},\bj^{(m)})})$ is finite, where we use the
shorthand  $\widehat{\mathcal{O}}(x_{\bi,\bj})\coloneqq\big(\mathcal{O}_{\bi}(x_{\bi,\bj})\cup
\mathcal{O}_{\bj}(x_{\bi,\bj})\big)$. Then the set
\begin{equation*}
A_r\coloneqq \bigcup_{\ell=0}^n \bigg( \big(f_{(\bi^{(r+1)})_{\ell}^1}\big)^{-1} \Big(
\bigcup_{m=1}^r \widehat{\mathcal{O}}(x_{(\bi^{(m)},\bj^{(m)})}) \Big) \cup
\big(f_{(\bj^{(r+1)})_{\ell}^1}\big)^{-1} \Big( \bigcup_{m=1}^r
\widehat{\mathcal{O}}(x_{(\bi^{(m)},\bj^{(m)})}) \Big) \bigg)
\end{equation*}
is also finite since all $f_{\bi}$ are strictly monotone (for $\ell=0$ it is defined to be identity map).

If $(\bi^{(r+1)},\bj^{(r+1)})$ is not
bad, then using that $A_r$ is finite and the continuity of the maps one can choose $\hat
x_{(\bi^{(r+1)},\bj^{(r+1)})}\in(0,1)\setminus A_r$ for which \cref{it:lla}
holds. By continuity of the maps, there exists a small neighbourhood of $\hat
x_{(\bi^{(r+1)},\bj^{(r+1)})}$ in $(0,1)\setminus A_r$ where \cref{it:lla} still holds, and by the same
argument as before can be used to pick a $x_{(\bi^{(r+1)},\bj^{(r+1)})}$ from this small
neighbourhood for which \cref{it:lla,it:llb,it:llc} all hold. If $(\bi^{(r+1)},\bj^{(r+1)})$ is bad, then
analyticity and our assumption imply again the existence of
$x_{(\bi^{(r+1)},\bj^{(r+1)})}\in(0,1)\setminus A_r$ that satisfies \cref{it:llb} which completes the
induction.
\end{proof}

\begin{proposition}\label{prop:AnalyticBumpFunc}
Let $f\in \mathcal{S}^{\omega}([0,1])$ with $f([0,1])\subseteq(0,1)$. There exists a constant
$C>0$ such that for any two finite collections of points
$\mathcal{Y}=\{y_1<\ldots<y_M\}\subseteq[0,1]^M$ and $\mathcal{Z}=\{z_1<\ldots<
z_Q\}\subseteq[0,1]^Q$ with $\mathcal{Y}\cap\mathcal{Z}=\varnothing$ we have the following: for
every $\delta>0$ and $\alpha>0$ there exists an analytic function
$g\in\mathcal{S}^{\omega}([0,1])$ such that
\begin{enumerate}[label=(\roman*)]
\item $g(z_i)=f(z_i)$ for every $z_i\in\mathcal{Z}$ and $g(y_i)=f(y_i)$ for every $y_i\in\mathcal{Y}$, moreover,
\begin{equation*}
\sup_{x\in[0,1]}|g(x)-f(x)|<\delta;
\end{equation*}
\item  $g'(z_i)=f'(z_i)$ for every $z_i\in\mathcal{Z}$ and $g'(y_i)=f'(y_i)$ for every $y_i\in\mathcal{Y}$, moreover,
\begin{equation*}
	\sup_{x\in[0,1]}|g'(x)-f'(x)|<\delta;
\end{equation*}
\item $g''(z_i)=f''(z_i)$ for every $z_i\in\mathcal{Z}$, however,
\begin{equation*}
	\left| \frac{g''(y_i)}{g'(y_i)} - \frac{f''(y_i)}{f'(y_i)} \right|\geq \alpha \quad\text{
	for every } y_i\in\mathcal{Y},
\end{equation*}
nevertheless, $	\sup_{x\in[0,1]}|g''(x)-f''(x)|<C\cdot \alpha +\delta$.
\end{enumerate}
\end{proposition}

We may assume that $\mathcal{Y}\neq\varnothing$, otherwise there is nothing to prove. By adjoining
an auxiliary point in $[0,1]\setminus\mathcal{Y}$ to $\mathcal{Z}$ if necessary, we may also assume
that $\mathcal{Z}\neq\varnothing$. Our claim is that with appropriate choices of
$a_1,\ldots,a_M>0$ and $\eta_1\ldots,\eta_M>0$ the analytic function
\begin{equation}\label{eq:BumpFunc}
g(x)\coloneqq f(x)\cdot e^{\varphi(x)\cdot \psi(x)\cdot A(x)},
\end{equation} 
where
\begin{equation*}
\varphi(x)=\prod_{y_i\in\mathcal{Y}} (x-y_i)^2,\; \psi(x) = \prod_{z_i\in\mathcal{Z}} (x-z_i)^4
\;\text{ and }\; A(x)= \sum_{i=1}^{M} a_i\cdot e^{\frac{-(x-y_i)^2}{\eta_i}}
\end{equation*}
satisfies the conditions of~\cref{prop:AnalyticBumpFunc}. We will often use the following simple fact.

\begin{lemma}
Let us fix constants $c, p, \delta>0$ and $q> -p/2$. Then
\begin{equation}\label{eq:BoundA(x)}
\sup_{x\in\mathbb{R}} c\cdot \sigma^q\cdot |x|^p\cdot e^{\frac{-x^2}{\sigma}} \leq \delta
\quad\text{ whenever } 0<\sigma \leq (2e/p)^{\frac{p}{2q+p}}\cdot (\delta/c)^{\frac{1}{q+p/2}}.
\end{equation}
\end{lemma}
\begin{proof}
It is easy to check that the global maximum of the function is at $x^2=p\sigma/2$. Substituting back
this value and using the upper bound on $\sigma$ gives the claim. 
\end{proof}

\begin{proof}[Proof of~\cref{prop:AnalyticBumpFunc}]
In this proof, we write $\|h\|\coloneqq\sup_{x\in I}|h(x)|$. We first choose the amplitudes $a_i$ to
obtain the lower bound in (iii), then take the widths $\eta_i$ sufficiently small to prove the norm
estimates. The function $g$ in \cref{eq:BumpFunc} is analytic wherever $f$ is analytic. We may
decrease $\delta$ to absorb fixed factors in the estimates below and to ensure that (i)
and (ii) imply $g([0,1])\subseteq(0,1)$ and $0<|g'(x)|<1$ on $[0,1]$.

Let us now calculate the derivatives of $g$. Clearly,
\begin{equation*}
	g'= f'\cdot e^{\varphi\cdot \psi\cdot A} + f\cdot e^{\varphi\cdot \psi\cdot A}\big(
	\varphi'\cdot \psi\cdot A + \varphi\cdot \psi'\cdot A +\varphi\cdot \psi\cdot A' \big),
\end{equation*}
where
\begin{align*}
	\varphi'(x)&= \sum_{i=1}^{M} 2(x-y_i) \prod_{y_j\in\mathcal{Y}\setminus\{y_i\}} (x-y_j)^2, \\
	\psi'(x)&= \sum_{i=1}^{Q} 4(x-z_i)^3 \prod_{z_j\in\mathcal{Z}\setminus\{z_i\}} (x-z_j)^4, \\
	A'(x)&= -2\cdot \sum_{i=1}^{M} a_i \frac{x-y_i}{\eta_i} e^{\frac{-(x-y_i)^2}{\eta_i}}.
\end{align*}
Moreover,
\begin{multline*}
	g'' = f''  e^{\varphi \psi A} + 2f'e^{\varphi \psi A}\big( \varphi' \psi A + \varphi \psi' A
	+\varphi \psi A' \big) + f e^{\varphi \psi A}\big( \varphi' \psi A + \varphi \psi' A
	+\varphi \psi A' \big)^2 \\
	+ f e^{\varphi \psi A}\big( \varphi'' \psi A + \varphi \psi'' A +\varphi \psi A'' +2\varphi'
	\psi' A + 2\varphi' \psi A' +2\varphi \psi' A'\big),
\end{multline*}
where
\begin{align*}
	\varphi''(x)&= 2\sum_{i=1}^{M}  \prod_{\substack{y_j\in\mathcal{Y}\\ y_j\neq y_i}}
	(x-y_j)^2+ 4\sum_{i=1}^{M} (x-y_i) \sum_{\substack{k=1\\k\neq i}}^{M} (x-y_k)
	\prod_{\substack{y_j\in\mathcal{Y} \\ y_j\notin\{y_i,y_k\}}} (x-y_j)^2, \\
	\psi''(x)&= 12\sum_{i=1}^{Q} (x-z_i)^2 \prod_{\substack{z_j\in\mathcal{Z} \\ z_j\neq z_i}}
	(x-z_j)^4 + 16\sum_{i=1}^{Q} (x-z_i)^3 \sum_{\substack{k=1\\k\neq i}}^{Q} (x-z_k)^3
	\prod_{\substack{z_j\in\mathcal{Z} \\ z_j\notin \{z_i,z_k\}}} (x-z_j)^4,  \\
	A''(x)&= 2\cdot \sum_{i=1}^{M} a_i \Big( \frac{2(x-y_i)^2}{\eta_i^2} -\frac{1}{\eta_i} \Big)
	e^{\frac{-(x-y_i)^2}{\eta_i}}. 
\end{align*}
By construction $\varphi(y_i) = \varphi'(y_i) =0$ for every $ y_i\in\mathcal{Y}$ and $\psi(z_i) =
\psi'(z_i) = \psi''(z_i)=0$ for every $z_i\in\mathcal{Z}$, hence,
\begin{equation*}
	g(y_i)=f(y_i) \text{ and } g'(y_i)=f'(y_i) \text{ for every } y_i\in\mathcal{Y},
\end{equation*}
furthermore,
\begin{equation*}
	g(z_i)=f(z_i),\, g'(z_i)=f'(z_i) \,\text{ and } g''(z_i)=f''(z_i) \;\text{ for every }
	z_i\in\mathcal{Z}.
\end{equation*}
Since $\varphi''(y_i)\neq 0$, let us evaluate
\begin{align*}
	g''(y_i) &= f''(y_i) + f(y_i)\varphi''(y_i)\psi(y_i)A(y_i) \\
	&=  f''(y_i) + 2f(y_i) \prod_{\substack{y_j\in\mathcal{Y}\\ y_j\neq y_i}} (y_i-y_j)^2
	\prod_{z_j\in\mathcal{Z}} (y_i-z_j)^4 \bigg( a_i+ \sum_{\substack{j=1 \\ j\neq i}}^{M}
	a_j\cdot e^{\frac{-(y_i-y_j)^2}{\eta_j}} \bigg).
\end{align*}
Dividing both sides by $f'(y_i)=g'(y_i)$ and rearranging we get
\begin{equation*}
	\left| \frac{g''(y_i)}{g'(y_i)} - \frac{f''(y_i)}{f'(y_i)} \right|\geq 2a_i\cdot
	\frac{|f(y_i)|}{|f'(y_i)|}\prod_{y_j\in\mathcal{Y}\setminus\{y_i\}} (y_i-y_j)^2
	\prod_{z_j\in\mathcal{Z}} (y_i-z_j)^4 \geq \alpha
\end{equation*}
for all $y_i\in\mathcal{Y}$ as required if we choose
\begin{equation}\label{eq:a_iChoice}
	a_i \coloneqq \frac{\alpha \cdot |f'(y_i)|}{2  \cdot |f(y_i)|}
	\Big(\prod_{y_j\in\mathcal{Y}\setminus\{y_i\}} (y_i - y_j)^2 \cdot \prod_{z_j\in\mathcal{Z}}
	(y_i - z_j)^4 \Big)^{-1}.
\end{equation}

It remains to bound the norms $\|g-f\|$, $\|g'-f'\|$ and $\|g''-f''\|$. Using that $|x-z_i|\leq 1$,
we have the trivial bounds
\begin{equation*}
\|\psi\| \leq 1, \quad \|\psi'\| \leq 4Q \quad\text{ and }\quad  \|\psi''\| \leq 16Q^2+12Q.
\end{equation*}
Similarly, using that $|x-y_i|\leq 1$, we also have for every $x\in I$
\begin{equation*}
|\varphi(x)| \leq \min_{y_i\in\mathcal{Y}} |x-y_i|^2, \quad |\varphi'(x)| \leq
2M\cdot\min_{y_i\in\mathcal{Y}} |x-y_i|  
\end{equation*} 
and
\begin{equation*}
|\varphi''(x)| \leq 4M^2\cdot \min_{y_i\in\mathcal{Y}} |x-y_i| + \bigg|2\sum_{i=1}^{M}
\prod_{\substack{y_j\in\mathcal{Y}\\ y_j\neq y_i}} (x-y_j)^2\bigg|.
\end{equation*}
Choose $\eta_i$ so small such that
\begin{equation*}
\eta_i \leq \min \bigg\{ \underbrace{2e\cdot \left(\frac{\delta}{2M^2(16Q^2+12Q)a_i}\right)^2
}_{(C1)}, \underbrace{ \left(\frac{\delta\cdot
e^{3/2}}{8MQa_i(3/2)^{3/2}}\right)^2}_{(C2)}\bigg\}.
\end{equation*}
Several norms can be handled simultaneously:
\begin{multline*}
\max\{ \|\varphi \psi A\|,\|\varphi' \psi A\|, \|\varphi \psi' A\|, \|\varphi \psi'' A\|, \|\varphi'
\psi' A\| \} \\
\leq \sup_{x\in I}\sum_{i=1}^M 2M(16Q^2+12Q)a_i |x-y_i| \cdot e^{\frac{-(x-y_i)^2}{\eta_i}} \leq \delta
\end{multline*}
by~\cref{eq:BoundA(x)} and $(C1)$ (with the choice $p=1, q=0$ and $c=2M(16Q^2+12Q)a_i$). Two more
norms can be handled together:
\begin{equation*}
\max\{ \|\varphi \psi A'\|, \|\varphi \psi' A'\| \}
	\leq \sup_{x\in I}\sum_{i=1}^M 8Qa_i \frac{|x-y_i|^3}{\eta_i} \cdot e^{\frac{-(x-y_i)^2}{\eta_i}} \leq \delta
\end{equation*}
by~\cref{eq:BoundA(x)} and $(C2)$ (with the choice $p=3, q=-1$ and $c=8Qa_i$). These bounds already
imply that $\|g-f\|\leq (e^{\delta}-1)\cdot \|f\|$ and $\|g'-f'\|\leq
(e^{\delta}-1)\cdot\|f'\|+3\delta e^{\delta} \cdot\|f\|$, which gives (i) and
(ii) after absorbing constant factors into $\delta$. By our choice of $\delta$, (i) and
(ii) give $g([0,1])\subseteq(0,1)$ and $0<|g'(x)|<1$ for every $x\in[0,1]$.
By analyticity and compactness, for some $\rho>0$ the map $g$ is analytic on $\mathcal{B}_{2\rho}$
and satisfies $0<|g'|<1$ on $\overline{\mathcal{B}_\rho}$. Since $g(I)\subseteq I$, the derivative
bound also gives $g(\overline{\mathcal{B}_\rho})\subseteq\mathcal{B}_\rho$. Hence
$g\in\mathcal{S}^\omega_\rho(I)\subseteq\mathcal{S}^\omega(I)$.

The remaining three norms, $\|\varphi'' \psi A\|,\|\varphi \psi A''\|$ and $\|\varphi' \psi A'\|$
require additional care. The trivial bounds can not be blindly used in some of the expressions when
$x$ is too close to one of the $y_i$. We demonstrate this on $\|\varphi'' \psi A\|$ and leave the
other two to the reader since the arguments are analogous.

Besides $\eta_i\leq\min\{(C1),(C2)\}$, we need further restrictions on $\eta_i$. Assume that
\begin{equation}\label{eq:eta1/3Bound}
	\eta_i\leq \min \bigg\{ \underbrace{ \Big( \frac{\delta\sqrt{2e}}{4M^3a_i} \Big)^2
	}_{(C3)}, \underbrace{ \frac{\delta e}{2M^2a_i} }_{(C4)} \bigg\},\qquad
	\eta_i^{1/3} < \frac{1}{2} \min\big\{ y_{i+1}-y_i, y_i-y_{i-1} \big\} 
\end{equation}
and
\begin{equation}\label{eq:etaOtherBound}
	\sum_{i=1}^M 2a_i\cdot e^{-\eta_i^{-1/3}} < \delta.
\end{equation}
Clearly all these conditions can be simultaneously satisfied. Using the bound on $\|\varphi''\|$,
\begin{equation*}
	\big|\varphi''(x)\cdot \psi(x)\cdot A(x)\big|\leq  \bigg| 2\psi(x)\cdot A(x)\cdot
	\sum_{i=1}^{M}  \prod_{\substack{y_j\in\mathcal{Y}\\ y_j\neq y_i}} (x-y_j)^2 \bigg| +
	4M^2\cdot \sum_{i=1}^{M} a_i|x-y_i| \cdot e^{\frac{-(x-y_i)^2}{\eta_i}}.
\end{equation*}
The second term is $\leq \delta$ because we can apply~\cref{eq:BoundA(x)} and $(C3)$. The first
term is a double sum which we split into two parts
\begin{equation*}
	\underbrace{2\psi(x)\cdot  \sum_{i=1}^{M} a_i e^{\frac{-(x-y_i)^2}{\eta_i}}
	\prod_{\substack{y_j\in\mathcal{Y}\\ y_j\neq y_i}} (x-y_j)^2}_{=: I(x)} +
	2 \sum_{k=1}^{M}  a_k  e^{\frac{-(x-y_k)^2}{\eta_k}} \underbrace{\sum_{\substack{i=1 \\
	i\neq k}}^{M} \prod_{\substack{y_j\in\mathcal{Y}\\ y_j\neq y_i}} (x-y_j)^2}_{\leq
      M(x-y_k)^2}.
\end{equation*} 
We can apply~\cref{eq:BoundA(x)} again to the second term and then $(C4)$ to see that the second
term is bounded above by $\delta$. What remains is to bound $I(x)$. This is where we distinguish whether
$x$ is close to a $y_i$ or not. Recall, we assume~\cref{eq:eta1/3Bound}. If $x$ is not too close to
any of the $y_i$ in the sense that $x\in\bigcap_{i=1}^M (y_i-\eta_i^{1/3}, y_i+\eta_i^{1/3})^C$,
then we use the trivial bounds
\begin{equation*}
	I(x) \leq \sum_{i=1}^{M} 2a_ie^{-\eta_i^{-1/3}} \stackrel{\eqref{eq:etaOtherBound}}{<} \delta.
\end{equation*}
So assume $x\in(y_i-\eta_i^{1/3}, y_i+\eta_i^{1/3})$ for some $y_i\in\mathcal{Y}$. Since $x$ is
still far enough from the other $y_j$, we just use the same bound there:
\begin{equation*}
	I(x) \leq 2a_i \psi(x) \prod_{y_j\in\mathcal{Y}\setminus\{y_i\}} (x-y_j)^2 +
	\sum_{\substack{j=1 \\ j\neq i}}^{M} 2a_j e^{-\eta_j^{-1/3}}
	\stackrel{\eqref{eq:etaOtherBound}}{<} \delta + 2a_i \psi(x)
	\prod_{y_j\in\mathcal{Y}\setminus\{y_i\}} (x-y_j)^2.
\end{equation*}
In the final term, we substitute the value of $a_i$ from~\cref{eq:a_iChoice} and $\psi(x)$ to get
\begin{align*}
	I(x)&\leq \delta+ \alpha \cdot \frac{|f'(y_i)|}{|f(y_i)|} \prod_{\substack{y_j\in\mathcal{Y} \\
	y_j\neq y_i}} \frac{(x-y_j)^2}{(y_i - y_j)^2}  \prod_{z_j\in\mathcal{Z}}
	\frac{(x-z_j)^4}{(y_i - z_j)^4} \\
	&\stackrel{\eqref{eq:eta1/3Bound}}{\leq}  \delta+  \alpha \cdot
	\frac{|f'(y_i)|}{|f(y_i)|}\prod_{\substack{y_j\in\mathcal{Y} \\ y_j\neq y_i}}\!\! \Big(
	1+\frac{\eta_i^{1/3}}{|y_i-y_j|} \Big)^2 \prod_{z_j\in\mathcal{Z}}\!\! \Big(
      1+\frac{\eta_i^{1/3}}{|y_i-z_j|} \Big)^4. 
\end{align*}
We may assume by choosing $\eta_i$ even smaller if necessary that the product of the final two
products is at most say 2. Since $f([0,1])\subseteq(0,1)$ and $0<|f'(x)|<1$ for
every $x$, we have shown that $I(x)\leq C\cdot\alpha+\delta$ for some constant $C>0$ depending
only $f$. This completes the bound for $\| \varphi'' \psi A\|$. 
\end{proof}

\subsection{Proof of~\cref{thm:ESCOpenDense}}

The main idea of the proof of \cref{thm:ESCOpenDense} is to apply~\cref{prop:AnalyticBumpFunc} to
each map $f_i$ of the IFS with appropriately chosen collections of points $\mathcal{Y}_i$ and
$\mathcal{Z}_i$ using~\cref{lem:Pointsx_ij} to get the maps $(g_i)_{i\in\mathcal{I}}$. We then lift
the IFS $(g_i)_{i\in\mathcal{I}}$ as in~\cref{eq:LiftedIFS} to obtain the dual IFS
$(G_i)_{i\in\mathcal{I}}$ and show that this
IFS satisfies the SSC. Then~\cref{thm:ESCOpenDense} follows immediately from~\cref{thm:DualSSC} and
\cref{lem:cont}.

\begin{proof}[Proof of \cref{thm:ESCOpenDense}]
By \cref{thm:DualSSC}, it is enough to show $\{\Phi\in\mathfrak{S}_N:\Phi^*\text{
satisfies the SSC}\}$ is open and dense in $\mathfrak{S}_N$ with respect to the metric $d_2$. The
set is open by \cref{lem:cont} and so it is enough to show that it is dense.

Let $r>0$ and $\Phi=(f_i)_{i\in\mathcal{I}}\in\mathfrak{S}_N$ be arbitrary but fixed. Choose
affine contractions $h_i$ with a common contraction ratio and pairwise disjoint images in $(0,1)$, and set
$f_i^t\coloneqq(1-t)f_i+th_i$. For small $t>0$, $(f_i^t)_{i\in\mathcal{I}}\in\mathfrak{S}_N$ and
$f_i^t(I)\subseteq(0,1)$. For distinct finite words $\bi,\bj$, the identity
$f_{\bi}^t\equiv f_{\bj}^t$ on $I$ holds for only finitely many $t\in[0,1]$, by analyticity in $t$
and its failure at $t=1$. Avoiding the countably many exceptional parameters, we obtain such an
IFS within $r/2$ of $\Phi$ in $d_2$ whose compositions along distinct finite words differ on $I$.
Replacing $\Phi$ by this IFS, it suffices to find $\Psi\in\mathfrak{S}_N$ with
$d_2(\Phi,\Psi)<r/2$ whose dual IFS satisfies the SSC.

Choose an admissible radius $\eps$ for $\Phi$. Let $C>0$ be a common constant in
\cref{prop:AnalyticBumpFunc} for $f_1,\ldots,f_N$.
Choose $k<K$ such that $F_i\overline{(k,K)}\subseteq(k,K)$ for every $i\in\mathcal{I}$, where
$\Phi^*=(F_i)_{i\in\mathcal{I}}$ is the dual IFS of $\Phi$.
By continuity of the endpoint functions on $I$, we can choose $\delta>0$ with $(C+3)\delta<r/2$
so that these inclusions persist for all IFSs within $(C+3)\delta$ of $\Phi$ in $d_2$.
Choose $n\geq1$ such that
$c_{\max}^n(K-k)<\delta/3$. Recall that,
$\mathcal{B}_n=\{(\bi,\bj)\in\Sigma_n\times\Sigma_n:\, i_1<j_1\}$. Let
$\{x_{\bi,\bj}:\, (\bi,\bj)\in\mathcal{B}_n\}$ be points as in~\cref{lem:Pointsx_ij}.

Recall $\mathcal{O}_{\bi}(x)= \{x,f_{i_1}(x), f_{i_2i_1}(x),\ldots,f_{i_{|\bi|}\ldots i_1}(x)\}$.
For every $i\in\mathcal{I}$, we define a partition of $\bigcup_{(\bi,\bj)\in\mathcal{B}_n}
\big(\mathcal{O}_{\bi}(x_{\bi,\bj})\cup
\mathcal{O}_{\bj}(x_{\bi,\bj})\big)$ consisting of two elements $\{\mathcal{Y}_i,
  \mathcal{Z}_i\}$ as follows:
\begin{itemize}
	\item $x_{\bi,\bj}\in\mathcal{Y}_{i_1}$ and $x_{\bi,\bj}\in\mathcal{Z}_{j_1}$ if
	  $(\bi,\bj)\in\mathcal{B}_n$ is bad;
	\item $x_{\bi,\bj}\in\mathcal{Z}_{i_1}$ and $x_{\bi,\bj}\in\mathcal{Z}_{j_1}$ if
	  $(\bi,\bj)\in\mathcal{B}_n$ is not bad;
	\item $x_{\bi,\bj}\in\mathcal{Z}_{\ell}$ for every $\ell\in\mathcal{I}\setminus\{i_1,j_1\}$,
	  and $y\in\mathcal{Z}_{\ell}$ for every $\ell\in\mathcal{I}$ and every point $y$ of
	  $\bigcup_{(\bi,\bj)\in\mathcal{B}_n} \big(\mathcal{O}_{\bi}(x_{\bi,\bj})\cup
	  \mathcal{O}_{\bj}(x_{\bi,\bj})\big)$ that is not one of the points $x_{\bi,\bj}$.
\end{itemize}
By \cref{lem:Pointsx_ij}, the sets $\mathcal{Y}_i$ and $\mathcal{Z}_i$ are well defined. We are now ready to
apply~\cref{prop:AnalyticBumpFunc}.

To each $f_i$ and $\mathcal{Y}_i, \mathcal{Z}_i$ we obtain a map
$g_i\in\mathcal{S}^\omega([0,1])$ which satisfies the properties
listed in~\cref{prop:AnalyticBumpFunc} with both parameters equal to $\delta$, and let
$\Psi=(g_i)_{i\in\mathcal{I}}\in\mathfrak{S}_N$. Then $d_2(\Phi,\Psi)<(C+3)\delta<r/2$.
Choose $0<\rho\leq\eps$ admissible for both $\Phi$ and $\Psi$, and take both dual IFSs on
$\mathcal{C}^\omega_\rho([0,1])$. The cylinder conditions on $I$ are unchanged, and we retain
the original $c_{\max}$ as a derivative bound for $\Phi$. We construct the dual IFS
$\Psi^*=(G_i)_{i\in\mathcal{I}}$ of $\Psi$ as in \cref{eq:LiftedIFS}, i.e.
\begin{equation*}
(G_ih)(x)\coloneqq g'_i(x)\cdot h(g_i(x)) + \frac{g''_i(x)}{g'_i(x)}.
\end{equation*}
By the choice of $\delta$, $G_i\overline{(k,K)}\subseteq(k,K)$ for every $i\in\mathcal{I}$.

By \cref{prop:AnalyticBumpFunc}, if $(\bi,\bj)\in\mathcal{B}_n$ is not bad then
$x_{\bi,\bj}\in\mathcal{Z}_{i_1}$ and $x_{\bi,\bj}\in\mathcal{Z}_{j_1}$. Hence,
\begin{equation*}
(G_{\bi}h)(x_{\bi,\bj})=(F_{\bi}h)(x_{\bi,\bj}) \text{ and }
(G_{\bj}h)(x_{\bi,\bj})=(F_{\bj}h)(x_{\bi,\bj})\text{ for every
}h\in\mathcal{C}_\rho^\omega([0,1]),
\end{equation*}
and in particular, $\mathrm{conv}\big((G_{\bi}k)(x_{\bi,\bj}), (G_{\bi}K)(x_{\bi,\bj})\big) \cap
\mathrm{conv}\big((G_{\bj}k)(x_{\bi,\bj}),
(G_{\bj}K)(x_{\bi,\bj})\big)=\varnothing$, or equivalently, $G_{\bi}(k,K)\cap G_{\bj}(k,K)=\varnothing$.

Let us now suppose that $(\bi,\bj)\in\mathcal{B}_n$ is bad, so that
$x_{\bi,\bj}\in\mathcal{Y}_{i_1}$ and $x_{\bi,\bj}\in\mathcal{Z}_{j_1}$. Then
\begin{equation*}
	(F_{\bi_2^{|\bi|}}h)(f_{i_1}(x_{\bi,\bj}))=(G_{\bi_2^{|\bi|}}h)(g_{i_1}(x_{\bi,\bj})) \text{
	and } (F_{\bj}h)(x_{\bi,\bj})=(G_{\bj}h)(x_{\bi,\bj})\text{ for every
	}h\in\mathcal{C}_\rho^\omega([0,1]),
\end{equation*}
and so,
\begin{equation}\label{eq:finstep}
  (G_{\bi}h)(x_{\bi,\bj})-(F_{\bi}h)(x_{\bi,\bj})=
  \frac{g_{i_1}''(x_{\bi,\bj})}{g_{i_1}'(x_{\bi,\bj})}
  -\frac{f_{i_1}''(x_{\bi,\bj})}{f_{i_1}'(x_{\bi,\bj})}
\end{equation}
for every $h\in\mathcal{C}^\omega_\rho([0,1])$. The right side is independent of $h$ and has
absolute value at least $\delta$. The two intervals in \cref{eq:DisjointCylinders} at
$x_{\bi,\bj}$ intersect and each has length at most $c_{\max}^n(K-k)<\delta/3$.
By \cref{eq:finstep}, the perturbation translates the first interval by a distance at least
$\delta$ and leaves the second fixed, making them disjoint.
This clearly implies that $G_{\bi}(k,K)\cap G_{\bj}(k,K)=\varnothing$. Finally,
\cref{thm:ESCOpenDense} follows by \cref{lem:SSCEequiv}.
\end{proof}

\section{Conjugation to self-similar IFS}\label{sec:ProofConjugation}

Fix $f\in\mathcal{S}^\omega([0,1])$ and choose an admissible radius $\eps>0$. Along the lines of
\cref{eq:H_i(x)}, let us define the map
\begin{equation}\label{eq:H}
\hat{H}_f(x)\coloneqq\sum_{k=0}^\infty\frac{f''}{f'}(f^{\circ k}(x))\cdot (f^{\circ k})'(x),
\end{equation}
where $f^{\circ k}$ denotes the self-composition of $f$ $k$-times. Analogously to the proof of
\cref{thm:H_iAnalytic}, $\hat{H}\in\mathcal{C}^\omega_\eps([0,1])$.

Let us begin the proof of \cref{thm:SubConjugation} with the following observation, which is
essentially a one-dimensional analytic instance of the Poincar\'e--Siegel linearisation theorem; see
\cite[Theorem~2.8.2]{KatokHasselblatt}. We give a self-contained proof tailored to our setting,
since we need the explicit form of the conjugating map~$g$ in terms of $\hat{H}_f$.

\begin{lemma}\label{lem:conj0}
	Let $f\in\mathcal{S}_\eps^\omega([0,1])$. Then for every $a,b\in\mathbb{R}$ with $b\neq 0$
	there exists a map $g\in\mathcal{C}^\omega_\eps([0,1])$ whose restriction to $I$ is invertible
	onto $g(I)$ and which satisfies $g''(x)=\hat{H}_f(x)g'(x)$, $g(p)=a$ and $g'(p)=b$,
	where $p$ is the unique fixed point of
	$f$ in $[0,1]$ and $\hat{H}_f$ is defined in \cref{eq:H}. 
	
	Moreover, for every $g\colon[0,1]\to\mathbb{R}$ such that $g''(x)\equiv \hat{H}_f(x)g'(x)$
	$$
	g(f(x))=f'(p)g(x)+g(p)(1-f'(p)).
	$$
\end{lemma}

\begin{proof}
	Using the analyticity of $\hat{H}$, we get that the map
	$g(x)=b\int_p^xe^{\int_p^z \hat{H}(y)dy}dz+a$ is in $\mathcal{C}^\omega_\eps([0,1])$
	with $g(p)=a$ and $g'(p)=b$.
	Since $g'(x)=b e^{\int_p^x\hat{H}(y)\,dy}$ has constant nonzero sign on $I$, $g|_I$ is invertible.
	
	Now, integrating $\hat{H}$	
	\[
	\int_p^x \hat{H}(y)dy=\sum_{k=0}^\infty\left(\log\left|f'(f^{\circ k}(x))\right|-\log
	|f'(p)|\right)=\log\left(\prod_{k=0}^\infty\frac{f'(f^{\circ k}(x))}{f'(p)}\right).
	\]
	Moreover, letting
	\[
	\hat{g}(x)\coloneqq\int_p^xe^{\int_p^z \hat{H}(y)dy}dz,
	\]
	and using the previous equation, we get that 
	\begin{equation}\label{eq:ghat}
	\hat{g}(x)=\lim_{n\to\infty}\frac{(f^{\circ n})(x)-p}{(f'(p))^n}
	\end{equation}
	is also analytic. 
	
	Using \cref{eq:ghat}, it is easy to see that $\hat{g}(f(x))\equiv f'(p)\hat{g}(x)$. Also,
	for any map $g$ with $g''(x)\equiv \hat{H}_f(x)g'(x)$, we have $g(x)\equiv
	g'(p)\hat{g}(x)+g(p)$ for every $x\in[0,1]$. Thus,
	\[
	g(f(x))=g'(p)\cdot \hat{g}(f(x))+g(p)=g'(p)\cdot f'(p) \hat{g}(x)+g(p)=f'(p)\cdot g(x)+g(p)(1-f'(p)),
	\]
	which had to be proven.
\end{proof}

\begin{proof}[Proof of~\cref{thm:SubConjugation}]
	Let $\Phi=(f_i)_{i\in\mathcal{I}}\in\mathfrak{S}_N$. The first assertion of
	\cref{thm:SubConjugation} follows from the construction in \cref{lem:conj0}, applied to each
	$f_i$ with $a=0$ and $b=1$.
	
	For the assertion \cref{it:conj2} of \cref{thm:SubConjugation}, choose a common admissible
	radius for $f_{\bi}$ and $f_{\bj}$. We apply \cref{it:conj1} to the sub-IFS
	$(f_{\bi},f_{\bj})$ when its attractor is not a singleton,
	using $\hat{H}_{f_{\bi}}=H_{(\bi_{|\bi|}^1)^\infty}$ by \cref{eq:H=f''/f',eq:H}.
	If $f_{\bi}$ and $f_{\bj}$ have a common fixed point and are conjugated to similarities,
	these similarities commute, so $f_{\bi\bj}=f_{\bj\bi}$. Since $\bi\bj\neq\bj\bi$ have the
	same length, this gives the required equality after reversing these words. Conversely,
	if $\hat{H}_{f_{\bi}}\equiv\hat{H}_{f_{\bj}}$, then \cref{lem:conj0} gives a common
	conjugating map. We finish the proof by showing \cref{it:conj1}.
	
	Let $\Phi=(f_i)_{i\in\mathcal{I}}\in\mathfrak{S}_N$. First, suppose that $H(x)\coloneqq
	H_{\bi}(x)\equiv H_{\bj}(x)$ for every $\bi,\bj\in\Sigma$, where $H_{\bi}$ is the dual
	natural projection defined in \cref{eq:H_i(x)}. Then $H(x)=H_{(i)^{\infty}}=\hat{H}_{f_i}$
	for every $i\in\mathcal{I}$. Hence, by \cref{lem:conj0} there exists
	$g\colon[0,1]\to\mathbb{R}$ analytic such that $g''(x)\equiv H(x)g'(x)$, and
	$g(f_i(x))=f_i'(p_i)g(x)+g(p_i)(1-f_i'(p_i))$ for every $i\in\mathcal{I}$, where
	$f_i(p_i)=p_i$.
	
	Finally, let us suppose that $\Phi=(f_i)_{i\in\mathcal{I}}\in\mathfrak{S}_N$ is conjugated
	to a self-similar IFS $(x\mapsto \lambda_ix+t_i)_{i\in\mathcal{I}}$ by the invertible
	analytic map $g\colon[0,1]\to\mathbb{R}$. Let $p_{\bi}$ be the fixed point of $f_{\bi}$ for
	every $\bi\in\Sigma_*$. 
	Then,
\[
g(f_{\bi}(x)) = \lambda_{\bi} g(x) + t_{\bi}
\quad\text{and}\quad
g(p_{\bi}) = \frac{t_{\bi}}{1-\lambda_{\bi}}.
\]
We have
\[
g'(f_{\bi}(x))\cdot f_{\bi}'(x) = \lambda_{\bi} g'(x)
\quad\text{and}\quad
|g'(p_{\bi})||f'_{\bi}(p_{\bi})-\lambda_{\bi}| = 0.
\]
The first identity shows that the zero set of $g'$ is mapped into itself by every $f_i$, so if
nonempty it contains the attractor of $\Phi$. Since the attractor is infinite, analyticity would
give $g'\equiv0$, contradicting the invertibility of $g$. Thus $g'$ has no zero in $[0,1]$.
Since $|g'(p_{\bi})|>0$ we must have $f'_{\bi}(p_{\bi}) = \lambda_{\bi}$.
Differentiating again we get
\[
g''(f_{\bi}(x))\cdot f'_{\bi}(x)^2 +g'(f_{\bi}(x))\cdot f''_{\bi}(x) = \lambda_{\bi} g''(x)\text{
for every }x\in[0,1]
\]
and by using \cref{eq:H=f''/f'}, we have that
\[
\frac{g''(p_{\bi})}{g'(p_{\bi})} =
\frac{f_{\bi}''(p_{\bi})}{f_{\bi}'(p_{\bi})(1-f'_{\bi}(p_{\bi}))}
=\frac{H_{\bi_{|\bi|}^1}(p_{\bi})}{1-f'_{\bi}(p_{\bi})}.
\]
Now, let $\bi,\bj\in\Sigma$ be arbitrary but fixed, and let $\bk_n = \bi_1^{n} \bj_n^1$. Then
$p_{\bk_n}\to \pi(\bi)$ as $n\to\infty$, where we recall that $\pi\colon\Sigma\to\mathbb{R}$ is the
natural projection of $\Phi$ defined in~\cref{eq:natProj}. Furthermore, by \cref{thm:H_iAnalytic,thm:difbound} and by
$f'_{\bk_n}(p_{\bk_n})\to0$ as $n\to\infty$
\[
\frac{g''(p_{\bk_n})}{g'(p_{\bk_n})}
=\frac{H_{\bj_1^n \bi_n^1}(p_{\bk_n})}{1-f'_{\bk_n}(p_{\bk_n})}
\to
H_{\bj}(\pi(\bi))\quad \text{as}\quad n\to\infty.
\]
However, the left-hand side converges to $g''(\pi(\bi))/g'(\pi(\bi))$, and so, we get 
$$
\frac{g''}{g'}(\pi(\bi))=H_{\bj}(\pi(\bi))
$$
for every $\bi,\bj\in\Sigma$. In particular, $H_{\bj}(\pi(\bi))=H_{\bk}(\pi(\bi))$ for every
$\bi,\bj,\bk\in\Sigma$. Since the attractor of $\Phi$ is not a singleton (i.e. uncountable), the
maps $H_{\bj}$ are analytic, we have that $H_{\bj}(x)\equiv H_{\bk}(x)$ for all $\bj,\bk\in \Sigma$.
\end{proof}

\begin{proof}[Proof of \cref{thm:DualConj}]
	The claim follows by \cref{thm:SubConjugation} and \cref{thm:H_iAnalytic} if the attractor of
	$\Phi$ is not a singleton. The proofs of \cref{it:conj2} and the reverse implication in
	\cref{it:conj1} do not use this assumption. Thus, it remains to show that if all $f_i$ fix
	$p$ and $\Phi$ is conjugated to a self-similar IFS, then $\Lambda^*$ is a singleton.

	The conjugated similarities have a common fixed point and commute, so the maps $f_i$
	commute. Let $\hat{g}$ be the map in \cref{eq:ghat} for $f_1$. By commutativity,
	\[
	\hat{g}(f_i(x))=\lim_{n\to\infty}
	\frac{f_i(f_1^{\circ n}(x))-p}{(f_1'(p))^n}=f_i'(p)\hat{g}(x)
	\]
	for every $i\in\mathcal{I}$. Since $\hat{g}'>0$, differentiating twice shows that
	$\hat{g}''/\hat{g}'$ is fixed by every $F_i$, so $\Lambda^*$ is a singleton.
\end{proof}

\printbibliography

\Addresses

\end{document}